\documentclass[a4paper]{article}

\title{Beginner's guide to Aggregation-Diffusion Equations}
\author{%
    David Gómez-Castro%
    	\thanks{Departamento de Matemáticas, Universidad Autónoma de Madrid. \href{mailto:david.gomezcastro@uam.es}{david.gomezcastro@uam.es}}
}

\usepackage[utf8]{inputenc}
\DeclareUnicodeCharacter{00A0}{ }

\usepackage[
url=false,
isbn=false,
backend=biber,
uniquename=init,
giveninits,
maxbibnames=100,
doi=false,
sortcites=true,
style=alphabetic,
maxalphanames=5,minalphanames=4,
sorting=nyt,
date = year,
eprint=false,
block=none]{biblatex}
\AtBeginRefsection{\GenRefcontextData{sorting=ynt}}
\AtEveryCite{\localrefcontext[sorting=ynt]}

\usepackage{csquotes}

\renewbibmacro{in:}{}
\DeclareFieldFormat[article]{citetitle}{#1}
\DeclareFieldFormat[article]{title}{#1}

\AtEveryBibitem{%
  \ifentrytype{online}
    {\clearfield{year}
    \clearfield{date}
    \clearfield{url}
    \clearfield{urldate}
    \clearfield{eprintclass}
    \clearfield{urlyear}
    \clearfield{pubstate}}
    {}%
}

\usepackage{amsmath,amssymb,mathtools,amsthm}
\usepackage{xcolor,graphicx,float,enumerate}

\usepackage[font=small]{caption}
\usepackage{subcaption}

\usepackage[hidelinks]{hyperref}

\numberwithin{equation}{section}

\usepackage[noabbrev,nameinlink]{cleveref}

\usepackage{esint}

\newtheorem{proposition}{Proposition}[section]
\newtheorem{theorem}[proposition]{Theorem}
\newtheorem{corollary}[proposition]{Corollary}
\newtheorem{lemma}[proposition]{Lemma}
\theoremstyle{definition}\newtheorem{definition}[proposition]{Definition}

\newtheorem{remark}[proposition]{Remark}
\newtheorem*{theorem*}{Theorem}

\allowdisplaybreaks[1]

\usepackage{amsmath,amssymb}

\newcommand{\ee}{\varepsilon}
\newcommand{\unk}{}

\newcommand{\calK}{{\mathcal K}}
\newcommand{\calF}{{\mathcal F}}
\newcommand{\Prob}{{\mathcal P}}

\newcommand{\R}{\mathbb R}
\newcommand{\Rd}{{\R^d}}
\newcommand{\N}{\mathbb{N}}

\newcommand{\indicator}[1]{1_{#1}}

\DeclareMathOperator{\diver}{div}

\DeclareMathOperator{\sign}{sign}
\DeclareMathOperator{\supp}{supp}

\newcommand*\diff{\mathop{}\!\mathrm{d}}
\newcommand{\defeq}{\vcentcolon=}

\newcommand{\dr}{\diff r}

\newcommand{\Wass}{{\mathfrak{W}}}
\newcommand{\calU}{\mathcal U}
\newcommand{\calW}{\mathcal W}

\DeclareMathOperator*{\argmin}{argmin}

\usepackage{fancyhdr}
\date{}

\usepackage[]{geometry}
\usepackage[
    final
]{changes}

\definechangesauthor[name={Reviewer 1},color=blue]{R1}
\definechangesauthor[name={Other comments},color=purple]{O}

\begin{document}

\maketitle

\begin{abstract}
	The aim of this survey is to serve as an introduction to the different techniques available in the broad field of Aggregation-Diffusion Equations.
	We aim to provide historical context, key literature, and main ideas in the field.
	We start by discussing the modelling and famous particular cases: Heat equation, Fokker-Plank, Porous medium, Keller-Segel, Chapman-Rubinstein-Schatzman, Newtonian vortex, Caffarelli-Vázquez, McKean-Vlasov, Kuramoto, and one-layer neural networks. In Section 4 we present the well-posedness frameworks given as PDEs in Sobolev spaces, and gradient-flow in Wasserstein. Then we discuss the asymptotic behaviour in time, for which we need to understand minimisers of a free energy. We then present some numerical methods which have been developed. We conclude the paper mentioning some related problems.
\end{abstract}

\bigskip

\setcounter{tocdepth}{1}
\tableofcontents

\comment{
	\textbf{Reply to the referees.} \\
	Comments in blue are replies to the referee. \\
	Comments in purple are additions suggested by other colleagues since the publication of the preprint
}

\newpage
		
\section{Introduction}

The aim of this survey is to serve as a general introduction to the many models and techniques used to study the \emph{Aggregation-Diffusion} family of equations
\begin{equation}
	\label{eq:ADE}	
	\tag{ADE}
	\frac{\partial \rho}{\partial t} =   \diver \Big( \rho \nabla (    U'(\rho)  
	+ 
	V 
	+ 
	W * \rho 
	) \Big).
\end{equation}
This survey does not intend to provide an encyclopedic presentation including all possible references and results in maximal generality (this would be impossible), but to give newcomers a general overview of the subject.

\section{Modelling}
Let us discuss first the modelling, and then we devote \Cref{sec:examples} to several modelling problems which belong to this family.
On the one hand, diffusion is usually modelled as a conservation law.

\subsection{Continuity equations}
Let $\rho$ be a density and $\omega \subset \Rd $ any control volume, if $\mathbf F$ is the out-going flux 
\begin{equation}
	\label{eq:conservation control volume}
	\frac{\diff }{\diff t} \int_\omega \rho \diff x =- \int_{\partial \omega} \mathbf F \cdot \mathbf n \diff S=  -\int_{ \omega} \diver \mathbf F \diff x
\end{equation} 
Hence we arrive at the \emph{continuity equation}
\begin{equation} 
	\label{eq:conservation law}
	\frac{\partial \rho}{\partial t} = -\diver \mathbf F.
\end{equation}

For the transport of heat, we use \textit{Fourier's law} to model the flux $\mathbf F =  -D \nabla \rho$ and yields the heat equation
\begin{equation*}
	\tag{HE}
	\label{eq:HE}
	\frac{\partial \rho}{\partial t} = D
	\Delta \rho .
\end{equation*}
A generalisation of this problem which is used to the flow a gas through porous media. We derive the model as mass balance by writing the flux in terms of a velocity $\mathbf v$, i.e., $\mathbf F = -\rho \mathbf v$; we use \textit{Darcy's law} to relate the velocity with the pressure $p$ as $\mathbf v = - \frac{k} \mu \nabla p$; and lastly we use a general \textit{state equation} $p = \phi(\rho)$, where for perfect gases $\phi (\rho) = p_0 \rho^\gamma$. Eventually, we recover that
$\mathbf F =  - \nabla \Phi (\rho )$ for some non-decreasing $\Phi : \R \to \R$, and this yields the so-called \emph{Porous Medium Equation}
\begin{equation*}
	\tag{PME$_\Phi$}
	\label{eq:PME Phi}
	\frac{\partial \rho}{\partial t} = 
	\Delta \Phi( \rho ).
\end{equation*}
It is common to select $ \Phi (\rho) = \rho^m$ for $m > 0$. The most complete reference for this equation is \cite{Vazquez2007}.
Notice $\Delta \Phi(\rho) =  \diver ( \Phi'(\rho) \nabla \rho  )$ so we take 
\begin{equation} 
	\label{eq:Phi and U}
	U'' (\rho) = \frac{\Phi' (\rho)}\rho.
\end{equation}

The classical modelling also allows for a transport velocity field $\mathbf v (t,x)$ to be included. For example, if a pollutant is being transported by water, then $\mathbf v(t,x)$ is the velocity of the water.
The flux of mass being moved is recovered from the density and the velocity as $\rho \mathbf v$. 
A general flux then looks like
$$
	\mathbf F = - \nabla \Phi(\rho) + \rho \mathbf v.
$$

\begin{remark}
	We do not discuss the case of sign-changing solutions. In that setting $u^m$ must be replaced by $|u|^{m-1} u$.
\end{remark}

\subsection{Particle systems}

We can also understand transport from the particle perspective. Consider a known velocity field $\mathbf v(t,x)$. 
Consider $N$ particles with positions $X_i (t)$ of equal masses $1/N$ moving in the velocity field
\begin{equation*}
	\frac{\diff X_i}{\diff t} = 
	\mathbf v(t, X_i(t))
	, \qquad i = 1, \cdots, N.
\end{equation*}
Notice that the evolution of the particles is decoupled.
To recover a PDE we consider the so-called
\emph{empirical distribution} defined as 
\begin{equation}
	\label{eq:empirical distribution}	 
	\mu_t^N = \sum_{j=1}^N \frac 1 N \delta_{X_j(t)} , 
\end{equation}
where $\delta_x$ is the Dirac delta at a point $x$.
First, notice that
\begin{align*}
	\frac{\diff}{\diff t} \int_\Rd \varphi(t,x) \diff \mu^N_t (x) &= \frac{\diff}{\diff t} \frac 1 N \sum_j \varphi(t, X_j(t)) \\
	&= \frac 1 N \sum_j \left( \frac{\partial \varphi}{\partial t}(t, X_j(t)) + \nabla \varphi(t,X_j(t)) \cdot \frac{\diff X_j}{\diff t} \right)\\
	&= \frac 1 N \sum_j \left( \frac{\partial \varphi}{\partial t}(t, X_j(t)) + \nabla \varphi(t,X_j(t)) \cdot \mathbf v(t, X_j(t)) \right)\\
	&= \int_\Rd  \left( \frac{\partial \varphi}{\partial t}(t, x) + \nabla \varphi(t,x) \cdot \mathbf v(t, x) \right) \diff \mu^N_t (x)
\end{align*}
If $\varphi \in C^1_c ((0,\infty) \times \Rd)$ (i.e., $\varphi(0,x) = 0$ and for $T$ large enough $\varphi(t,x) = 0$ for $t \ge T$) we integrate in time to recover
\begin{equation*}
	\int_0^\infty \int_\Rd \Big( \frac{\partial \varphi}{\partial t} + \mathbf v  \cdot \nabla \varphi \Big ) \diff \mu^N_t \diff t= 0.
\end{equation*}
Integrating by parts, we observe that this is the distributional formulation of the  
\emph{continuity equation}
\begin{equation*}
	\partial_t \mu^N +
	\diver (\mu^N  \mathbf v ) = 0.
\end{equation*}

\bigskip 

Aggregation and Confinement can be modelled by considering \emph{interacting particles} where the velocity field is aware of other particles
\begin{equation*}
	\frac{\diff X_i}{\diff t} =  
	-
	\sum_{\substack{ j=1 \\ j \ne i}}^N 
		\frac 1 N \nabla W (X_i - X_j)
		- \nabla V (X_i)
	, \qquad i = 1, \cdots, N.
\end{equation*}
The first term models the interactions between particles, with attract/repel each other as a function of their relative position, and the second term models a field attracting/repelling every particle.   
Noticing that 
$$
\sum_{\substack{ j=1 \\ j \ne i}}^N 
\frac 1 N \nabla W (x - X_j(t)) = \int_\Rd \nabla W (x-y) \diff \mu^N_t (y)
$$
we deduce from the previous reasoning that 
the empirical measure
$\mu^N$ is a distributional solution of the \textit{Aggregation-Confinement Equation}
\begin{equation}
	\partial_t \mu =  
	\diver (\mu   \nabla (  W * \mu +  V   ) ).
\end{equation}
Linear diffusion can be added to the particle system by introducing noise and working with Stochastic ODEs and a mean-field approximation (see \Cref{sec:McKean-Vlasov}).

\bigskip 

As we will see below, special attention has been paid to the case $V = 0$ which is simply called \textit{Aggregation Equation}
\begin{equation*}
	\tag{AE}
	\label{eq:AE}
	\frac{\partial \rho}{\partial t} = \diver (\rho \nabla W * \rho).
\end{equation*}
As a toy model we will also discuss the \textit{confinement problem}
\begin{equation*} 
	\tag{CE}
	\label{eq:ADE U,W = 0}
	\frac{\partial \rho}{\partial t} = \diver(\rho \nabla V ).
\end{equation*} 

\bigskip

Joining the many-particle approximation for aggregation with the Porous Medium diffusion we recover \eqref{eq:ADE}.
If we look for solutions defined in the whole space, then it is natural to consider only the case of finite mass. Hence, we assume $\rho \in L^1 (\Rd)$ and, in some sense,
\begin{equation}
	\label{eq:BC Rd}
		\rho_t(x)  \to 0  \text{ as } |x| \to \infty. \\
\end{equation}

\subsection{Boundary conditions in bounded domains of \texorpdfstring{$\Rd$}{Rd}} 

The problems above are often posed in the whole space $\Rd$. However, sometimes it is convenient to restrict them to bound domains, specially for the numerical analysis. Here and below we will denote by $\Omega$ an open and bounded set of $\Rd$ with a smooth boundary. Although it is possible to study these kinds of problems in less smooth domain (Lipschitz boundary, interior point conditions, etc…) this introduces additional difficulties we will not discuss.
Since with this kind of equation we intend to model systems that conserve mass, there are two natural approaches. 

\subsubsection{No-flux conditions}

We can stablish that no mass will cross the boundary of $\Omega$, $\partial \Omega$. Going back to \eqref{eq:conservation control volume}, a natural way is to establish no-flux conditions $\mathbf F \cdot \mathbf n = 0$ on $\partial \Omega$. With our choice of flux this means
\begin{equation*}
	\tag{BC}
	\label{eq:BC Omega}
	\rho \nabla (  U'(\rho) + V + W*\rho   ) \cdot \mathbf n = 0  \qquad \text{ for } t > 0 \text{ and } x \in \partial \Omega
\end{equation*}
where again $\rho_0$ is known, and we have to be a little careful with the notion of convolution when $\rho$ is only defined in $\Omega$. For this, we consider the extended notion
\begin{equation*}
	W*\rho \equiv W* \widetilde \rho, \qquad \text{where } \widetilde \rho = \begin{dcases}
		\rho & \text{in } \Omega, \\
		0 & \text{elsewhere}.
	\end{dcases}
\end{equation*}

In fact, the problem \eqref{eq:ADE} with general $V$ and $W$ in a bounded domain is rather delicate in terms of a priori estimates (see \Cref{sec:Lp estimate}). It is sometimes preferable to discuss the generalisation
\begin{equation*}
	\tag{ADE$^*$}
	\label{eq:ADE Omega}
	\frac{\partial \rho}{\partial t} =   \diver \Big( \rho \nabla (    U'(\rho)  
		+ 
		V 
		+ 
		\mathcal K \rho
		) \Big), 
		\qquad 
		\text{where }
		\mathcal K \rho (x) = \int_\Omega K(x,y) \rho(y) \diff y, 
\end{equation*}
and with the corresponding boundary conditions
\begin{equation*}
	\tag{BC$^*$}
	\label{eq:BC Omega star}
	\rho \nabla (  U'(\rho) + V + \mathcal K \rho   ) \cdot \mathbf n = 0 , \qquad \text{on } \partial \Omega.
\end{equation*}
If $\Omega$ is bounded we will choose 
\begin{equation} 
	\label{eq:V and K boundary condition}
	\nabla V \cdot \mathbf n = 0 \text{ in } \partial \Omega, \qquad \text{ and } \qquad 
	K \in C_c (\Omega \times \Omega) .
\end{equation}
A very typical example is 
\begin{equation}
	K(x,y) = \eta(x) W(x-y) \eta(y)
\end{equation}
with $\eta \in C^\infty_c (\Omega)$.

\subsubsection{Periodic boundary conditions}

Another approach is to stablish that any mass that crosses the boundary will ``come back'' on the opposite side. This happens, for example, if all data are periodic.
Some papers have studied periodic solutions defined over the unit cube $[0,1]^d$ (or, equivalently, the torus $\mathbb T^d$). In $d = 1$ this is equivalent to studying solutions defined over the unit circle $\mathbb S^1$. 
This is the case of the Kuramoto model described below.
An example of the case $d = 2$ can be found in \cite{carrillo2020LongTimeBehaviourPhase}.

\subsection{Free energy}

We conclude this section by introducing a functional which will be very useful below. We will show how 
\eqref{eq:ADE} is related to the free energy
\begin{equation*}
	\tag{FE}
	\label{eq:free energy}
	\mathcal F(\rho) =\int_\Omega U(\rho)  + \int_\Omega V \rho +  \frac 1 2 \iint_{\Omega \times \Omega } W(x-y) \rho(x) \rho(y) \diff x \diff y ,
\end{equation*} 
whereas \eqref{eq:ADE Omega} is related to the more general free energy
\begin{equation*}
	\tag{FE$^*$}
	\label{eq:free energy star}
	\mathcal F(\rho) =\int_\Omega U(\rho)  + \int_\Omega V \rho +  \frac 1 2 \iint_{\Omega \times \Omega } K(x,y) \rho(x) \rho(y) \diff x \diff y .
\end{equation*} 
For convenience, let us define the function
\begin{equation}
	\label{eq:first variation}
	\frac{\delta \mathcal F}{\delta \rho} [\rho] \defeq U'(\rho) + V + \int_\Omega K(\cdot,y)\rho(y) \diff y.
\end{equation}
This is called \textit{first variation of the free energy}, and  we will explain its importance below in \Cref{sec:free energy dissipation}.
Hence, \eqref{eq:ADE Omega} can be re-written as
\begin{equation}
	\label{eq:Wasserstein grad flow}
	\frac{\partial \rho}{\partial t} = \diver\left(\rho \nabla \frac{\delta \mathcal F}{\delta \rho}\right)
\end{equation} 
for $\calF$ given by \eqref{eq:free energy}.

Occasionally, it will be convenient for us to use the free energies of each of the terms. Hence, we define
\begin{equation}
    \label{eq:free energy U} 
    \calU[\rho] = \int_\Rd U(\rho), \qquad \mathcal V[\rho] = \int_\Rd V \rho, 
	\qquad
	\mathcal W[\rho] = \frac 1 2 \iint_\Rd W(x-y) \rho(x) \rho(y) \diff x \diff y.
\end{equation} 

\subsection{A comment on mass conservation}

Our evolution problem need, of course, an initial state $\rho_0$. 
Throughout this survey we will usually restrict ourselves non-negative unit mass initial data, i.e., 
\begin{equation*} 
	\tag{U$_0$}
	\label{eq:unit mass initial}
	\int_{\Rd} \rho_0(x) = 1.
\end{equation*}
All the equations presented above are of ``conservation type''. This means we expect preservation of mass, i.e., that provided unit initial mass \eqref{eq:unit mass initial} then there will be non-negative unit initial mass for all times, i.e., 
\begin{equation*}
	\tag{U}
	\label{eq:unit mass}
	\int_\Rd \rho_t = 1 , \qquad \forall t > 0.
\end{equation*}

\section{Famous particular cases}
\label{sec:examples}

\subsection{Heat equation and Fokker-Plank}

As explained above, the heat equation \eqref{eq:HE} is the most classical example in our family of equations.
It corresponds to \eqref{eq:ADE} with the choices $U = \rho \log \rho$ and $V = W = 0$.   
In $\Rd$ 
It is well known that \eqref{eq:HE} admits the Gaussian solution
$$
    K_t (x) = (4 \pi t)^{-\frac d 2} \exp \left( - \frac{|x|^2}{4t}   \right).
$$
All solutions (with suitable initial data) are precisely of the form
$$
	\rho_t = K_t * \rho_0,
$$
and it is not difficult to check that
$$
	\|\rho_t\|_{L^1} = M , \qquad \| \rho_t - M K_t \|_{L^1} \to 0 \quad{ as } t \to \infty.
$$
A good survey on the matter can be found in \cite{Vazquez2017}. 

\medskip

The fundamental solution $K_t$ is of the so-called \emph{self-similar} type
\begin{equation*} 
	\tag{SS}
	\label{eq:self-similar}
    K_t (x) = A(t) F\left( \frac x {\sigma(t)} \right).
\end{equation*} 
In this expression, $F$ is usually called self-similar profile and $\sigma(t), A(t)$ the scaling parameters. Due to mass conservation the natural choice is $A(t) = \sigma(t)^{-d}$. Plugging \eqref{eq:self-similar} into \eqref{eq:HE} and setting the total mass $\int K_t = 1$ yields the Gaussian. Since \eqref{eq:HE} is linear, $MK_t$ is a solution of mass $M$.
Usually, the self-similar solution has $\sigma(t)$ increasing, and it only makes sense to consider $\sigma(t) \ge 0$. Hence, we make the choice $\sigma(0) = 0$ that corresponds to $K_t = \delta_0$, a Dirac delta at $0$. This type of solutions with $\rho_0 = \delta_0$ are called \textit{source-type solutions}. 

\medskip 

This nice convolutional representation is not available in more general settings. There are several other ways to study the asymptotic behaviour.
One of the key tricks is the study of the self-similar change of variable 
$$
\rho(x,t) = u(y,\tau) (1 + t)^{-\frac d 2}, \qquad 2\tau = \log(1 + t),  
\added[id=R1]{%
	\qquad y = (1+t)^{-\frac 1 2} x,
}%
$$
which leads to the so-called linear Fokker-Planck equation
\begin{equation*} 
\tag{HE-FP}
\label{eq:Fokker linear}
\frac{\partial u}{\partial \tau} = \Delta_y u + \diver (u y ).
\end{equation*}
This equation corresponds to 
$$U_1 (\rho) = \rho \log \rho,$$
$V = \frac{|x|^2}{2}$ and $W = 0$.
Notice that stationary states can be obtained as solutions of the equation for the flux $\log u + \frac{|x|^2}{2} = -h$. This yields the Gaussian profile
\begin{equation}
	\label{eq:Gaussian}
	\widehat u(y) = AG(y), \qquad G(y) = \frac{1}{\sqrt{2\pi}} e^{-\frac{|y|^2}{2}}.
\end{equation}
The constant $A$ is the unique value such that $\int \rho_t = M$. In original variable, we recover $M K_t$.

\subsection{Porous medium equation}

As presented in the introduction, the most common example of \eqref{eq:PME Phi} is the Porous Medium Equation
\begin{equation*}
	\tag{PME}
	\label{eq:PME}
	\frac{\partial \rho}{\partial t} = \Delta \rho^m .
\end{equation*}
The range $m \in (0,1)$ is sometimes called \emph{Fast Diffusion Equation}.
The \eqref{eq:PME} corresponds, for $m \ne 1$, to 
$$U_m (\rho) = \frac{\rho^m}{m-1}$$ 
and $V = W = 0$. Notice that as $m \to 1$ we have $\rho^{m-1}/(m-1)$ tends to $\log \rho$ and hence $U_m$ tends to
$U_1$.
Notice that, due to the homogeneity of the equation, after rescaling space and time, we can preserve the equation (without introducing constants) and can work with \eqref{eq:unit mass initial}.

\bigskip 

Similarly to the heat equation, a suitable change of variable (detailed below in \Cref{sec:self-similar PME}) converts \eqref{eq:PME} into
\begin{equation*} 
	\tag{PME-FP}
	\label{eq:Fokker PME}
	\frac{\partial u}{\partial t} = \Delta u^m  + \diver (u x)
\end{equation*}
This corresponds to $U = U_m$ and $V = \frac{|x|^2}{2}$ and $W = 0$.
\eqref{eq:Fokker PME} admits a stationary solution given by the so-called Barenblatt profiles
\begin{equation} 
	\label{eq:Barenblatt profile}
	\hat u(x) = F_B(x) , \qquad F_B(x) \defeq \left(C - \tfrac{m-1}{2m}{|x|^2}\right)^{\frac{1}{m-1}}_+ ,
\end{equation}
where $C \ge 0$.
For $m > 1$ they have compact support, and for $m \in (0,1)$ they have power-like tails. To check whether finite-mass states are available we compute
$$
	\int_\Rd F_B \approx c + \int_1^\infty r^{\frac{2}{m-1}} r^{d-1} \dr .
$$	
Hence, finite mass steady states are only when $m > \frac{d-2}{d}$. This number is critical, as well will see below.
Undoing the change of variables we arrive at the self-similar solution
\begin{equation*} 
    \tag{ZKB}
	\label{eq:Barenblatt}
    B_t = t^{-\alpha} (C - k|x|^2t^{-2\beta})^{\frac 1 {m-1}}_+, \qquad \text{where } \alpha = \tfrac{d}{d(m-1)+2},\, \beta = \tfrac{\alpha}{d},\, k = \tfrac{\alpha(m-1)}{d}.
\end{equation*}
This solution is usually denoted ZKB after Zel’dovich and Kompaneets, and Barenblatt (see \cite{Vazquez2007} and the references therein). Usually this is shortened to Barenblatt.
When $m > \frac{d-2}{d}$ there is a unique constant $C$ so that \eqref{eq:unit mass}. It can be explicitly computed (see \cite[Section 17.5]{Vazquez2007}).

\bigskip

As we let $m \searrow \frac{d-2}{d}$ we get $C \nearrow \infty$, so for $m = \frac{d-2}{d}$ we expect $B_t = \delta_0$. 
In fact, Brezis and Friedman \cite{BrezisFriedman1983} showed that if $m \le \frac{d-2}{d}$ and $\rho_0^{(j)} \in L^1 (\Rd)$ are a sequence of initial data converging to $\delta_0$, then the corresponding solutions converge to $\rho^{(j)} (t,x) \to \delta_0 (x) \otimes 1(t)$. Therefore, \emph{very fast diffusion cannot diffuse a $\delta_0$}. This is because the diffusion coefficient is $m\rho^{m-1}$. For $m \ll 1$ this means that small values of $\rho$ are indeed diffused fast (hence the name of ``fast diffusion''). However, in regions where $\rho$ is large the diffusion is slow.

\bigskip

It is also worth pointing to the equivalent formulation in the so-called pressure variable $u = U'(\rho)$. We recover the Hamilton-Jacobi type equation
$$
	\frac{\partial u}{\partial t} = \Phi'(\rho) \Delta u + |\nabla u|^2,
$$
where we recall the relation \eqref{eq:Phi and U}.
Notice that $\Phi'(\rho)$ can be written as a non-linear function of $u$.
This formulation is convenient when $m > 1$ since $u = \frac{m}{m-1} \rho^{m-1}$. This presentation is much better for the study of viscosity solutions 
(see \cite{CaffarelliVazquez1999,BrandleVazquez2005}). 

\bigskip

First results of asymptotic behaviour for $m > 1$ were obtained by regularity, self-similarity, and comparison arguments in the late 70s and early 80s (see \cite{friedman1980AsymptoticBehaviorGas} and the references therein). 
Self-similar analysis and relative entropy arguments were later developed to improve the theoretical understanding (see \Cref{sec:self-similar PME,sec:relative entropy PME}). 
For a detailed explanation of the theory see \cite{Vazquez2006,Vazquez2007}.

\begin{remark}
	\label{rem:maximum principle heat}
	\eqref{eq:HE} and \eqref{eq:PME} are ``diffusive''. This can be seen in different directions. For example, at a point $(t_0,x_0)$ of local maximum of $\rho$ we get $\Delta \rho^m \le 0$ and hence $\frac{\partial \rho}{\partial t} \le 0$, i.e., \emph{the maximum values decays}. 
	This is strongly related to the \textit{maximum principle}.
	The converse happens to the minimum if it exists (notice that non-negative solutions in $\Rd$ will tend to $0$ as $|x| \to \infty$). 
\end{remark}

\begin{remark}
	\label{eq:PME numerology 1}
	Let us recall some ``numerological values''. The value $m = 1-\frac 2 d$ is the critical value for finite-mass Barenblatt. 
	It is also be the critical value below which there is finite-time extinction (see \Cref{sec:conservation of mass}). 
	The \eqref{eq:Barenblatt profile} has finite $\lambda$-moment, i.e.,
	$$\int_\Rd |x|^\lambda B_t(x) \diff x < \infty$$
	if and only if $m > 1 - \frac{2}{d+\lambda} $.
	\added[id=R1]{%
	In \Cref{sec:convexity} we will present a third relevant value $m = 1 - \frac 1 d$ coming from the Wasserstein gradient-flow analysis.
	}%
\end{remark} 

\begin{remark}
	We will not discuss here the limit $m \to 0$. In this direction there are two possibilities. 
	On the one hand, some authors have discussed the rescaled limit $\frac{u^m}{m} \to \log u$, and we arrive at the logarithmic diffusion equation $\partial_t u = \Delta \log(u)$. We point to \cite[Section 8.2] {Vazquez2006} for a discussing on this topic and \cite[Section 8.4] {Vazquez2006} for historical notes.
	On the other hand, some authors choose to look at the limit
	$|u|^{m-1} u \to \sign(u)$ leads to the so-called Sign Fast Diffusion Equation
	$\partial_t u = \Delta \sign(u)$. Some results in this direction are available \cite{BonforteFigalli2012}.
\end{remark}

\subsection{Transport problem with known potential}
\label{sec:transport with known potential}

To understand the transport terms in the \eqref{eq:ADE} family, we will discuss the toy problem
$$
	\frac{\partial \rho}{\partial t} = \diver (\rho \nabla V)
$$
where $V$ is known and does not depend on $t$. 
Imagine that $\rho_0 \ge 0$ is radially symmetric function. 
Computing the balance of mass of a ball $B_r$ we get
\begin{equation}
	\label{eq:transport known potential}
	\frac{\diff}{\diff t}\int_{B_r} \rho_t = \int_{B_r} \diver (\rho \nabla V) = \int_{\partial B_r} \rho \nabla V \cdot x.
\end{equation}
If $\nabla V \cdot x \le 0$ all balls $B_r$ are ``loosing mass'', i.e., then mass if always flowing ``away'' from $0$ and towards infinity. Hence, we can think of these kinds of potentials as ``diffusive'' (but not as strongly as in the sense of \Cref{rem:maximum principle heat}). 
We will show below by explicitly solving these problems that potentials such that $\nabla V \cdot x \ge 0$ will attract mass towards $0$, asymptotically (or even instantaneously) creating a singularity.

\subsection{Keller-Segel model}
\label{sec:Keller-Segel}

The Keller-Segel \cite{keller1970InitiationSlimeMold} (or Patlak-Keller-Segel model \cite{patlak1953RandomWalkPersistence}) model for describes the motion of cells by chemotactical attraction by means of the coupled system
\begin{equation*}
	\tag{KSE} 
	\label{eq:Keller-Segel un-normalised}
	\begin{dcases}
		\frac{\partial \rho}{\partial t} = \Delta \rho - \diver (\rho \nabla V_t ) \\ 
		-\Delta V_t= \rho_t .
	\end{dcases}
\end{equation*}
This model was first studied for $d = 2$. 
Some authors replace the second equation by a more general evolutionary problem $\ee \frac{\partial V_t}{\partial t} - \Delta V_t + \alpha V_t = \rho_t$. 
We will discuss \eqref{eq:Keller-Segel un-normalised}, which can be thought as the limit $\ee, \alpha \to 0$.
This simplification was first introduced by \cite{nanjundiah1973ChemotaxisSignalRelaying} to model the case where the diffusion of the chemical is much faster than its production.

\bigskip

So far, this model is not in the \eqref{eq:ADE} family. However, the can analysis the equation for $V_t$ to recover some information.
In $\Rd$ we can use the Fourier transform $\mathrm F$ for the elliptic problem to recover
$$
	|\xi|^2 \mathrm{F} [V_t] (\xi) = \mathrm{F}[\rho_t]
$$
and hence we can solve and recover
$
	V_t = W_{\rm N} * \rho 
$,
where $W_{\mathrm N} = \mathrm{F}^{-1}[|\xi|^{-2}]$. In fact, we have the closed expression
\begin{equation}
	\label{eq:Newtonian potential} 
	W_{\mathrm N} (x) = 
	\begin{dcases}
		-\frac{1}{2\pi} \log|x| & \text{if } d = 2, \\
		\frac{1}{d(d-2)\omega_d} |x|^{2-d} & \text{if } d > 2.
	\end{dcases}
\end{equation} 
Notice that $-\Delta W_{\rm N} = \delta_0$.
The so-called \textit{Newtonian potential} corresponds to $\rm N = -W_{\rm N}$.
Eventually, we can write this problem as \eqref{eq:ADE} where $U = \rho \log \rho$, $V = 0$ and $W = -W_{\rm N}$.

We can think about the convolution term $W*\rho_t$ as an aggregation potential. The justification is the following. If $\rho_t$ is non-increasing, then $V_t = W*\rho_t$ is also non-increasing, and hence $\nabla V_t \cdot x \le 0$. Going back to \Cref{sec:transport with known potential}, this is what we called aggregation potentials. Unfortunately, we do not have in general the sign $\nabla V_t \cdot x \le 0$ for all $t,x > 0$. In fact, \eqref{eq:Keller-Segel un-normalised} does not ``prefer'' $x = 0$ from other points.

Changing the sign of the aggregation to $\frac{\partial \rho}{\partial t} = \Delta \rho + \diver (\rho \nabla V_t )$ makes the problem more regularising. In \cite{biler1992ExistenceAsymptoticsSolutions} the author proves global well-posedness of that problem for $d = 3$ in a bounded domain with no-flux condition was shown to have.

\bigskip 

Sometimes it is convenient to express the problem in a single equation using operator notation:
$$		
    \frac{\partial \rho}{\partial t} = \Delta \rho - \diver (\rho \nabla (-\Delta)^{-1} \rho). 
$$ 
When written in this form, the authors usually allow for any initial mass $\|\rho_0\|_{L^1}$. We can rescale to impose \eqref{eq:unit mass initial}.
If we let $\chi = \|\rho_0\|_{L^1}$ and $u(t,x) = \rho(t,x) / \chi$ we arrive at the equation
\begin{equation*} 
	\tag{KSE$_\chi$} 
	\label{eq:Keller chi}
	\frac{\partial u}{\partial t} = \Delta u - \chi \diver (u \nabla (-\Delta)^{-1} u). 
\end{equation*}

\bigskip 

In bounded domains some authors introduce Neumann boundary conditions $\frac{\partial u}{\partial n} = \frac{\partial V}{\partial n} = 0$. 
The authors of \cite{JaegerLuckhaus1992} extend the problem to $-\Delta V = \beta u - \mu$ which a few authors later called Jäger-Luckhaus problem (e.g., \cite{senba2005BlowupBehaviorRadiala}). This is equivalent to replacing $(-\Delta)^{-1}$ by $(-\Delta + \rm{id})^{-1}$.

\bigskip 

Some authors later introduced the Keller-Segel model with Porous Medium Diffusion. One of the model arguments is that $m > 1$ would deal with ``over-crowding'' (i.e., the formation of Dirac deltas). We arrive at the problem
\begin{equation*}
	\tag{KSE$_{m,\chi}$} 
	\label{eq:Keller PME}
		\frac{\partial \rho}{\partial t} = \Delta \rho^m - \chi \diver (\rho \nabla (-\Delta)^{-1} \rho). 
\end{equation*}
This model has the advantage of not having finite time blow-up (see \cite{CalvezCarrillo2006}).
Existence and uniqueness of solutions for $m > 1$ and a family of $W$ covering $W_{\rm N}$ can be found in \cite{bedrossian2011LocalGlobalWellposednessa}. 
Some authors replace $(-\Delta)^{-1}$ by $(-\Delta + \gamma \rm{id})^{-1}$, see, e.g., \cite{luckhaus2007AsymptoticProfileOptimal,sugiyama2006GlobalExistenceDecay,sugiyama2007ApplicationBestConstant,sugiyama2007TimeGlobalExistence}.

As we will discuss below, \eqref{eq:Keller PME} admits global solutions in the sub-critical range 
\added[id=R1]{%
	$m > 2-\frac 2 d$.  However, for $m < 2-\frac 2 d$
}%
there is finite-time blow-up. In the critical case blow-up depends on $\chi$. Curiously enough, the first case studied was $d = 2$ and $m = 1$, which is critical. A good explanation of this dichotomy is the free-energy minimisation argument which we will provide in \Cref{sec:existence of minimisers}.

\subsection{Chapman-Rubinstein-Schatzman}

In the context of superconductivity \cite{chapman1996MeanfieldModelSuperconducting} introduced
a mean-field model for ``the motion of rectilinear vortices in the mixed state of a type-II semiconductor''. They argue this is the limit where the Ginzburg-Landau parameter $\kappa$ tends to $+\infty$ and the number of vortices becomes large one recovers the problem
\begin{equation*}
	\tag{CRSE}
	\label{eq:CRSE}
	\begin{dcases}
		\frac{\partial \rho}{\partial t} = \diver( |\rho| \nabla V_t ) \\
		-\Delta V_t + V_t = \rho_t .
	\end{dcases}
\end{equation*}
The model was first introduced in $\R^2$. Then, the authors also make a particularisation to a bounded domain $\Omega$ where they introduce Dirichlet conditions for $V_t(x) = h_0$ on $\partial \Omega$, and flux conditions setting the number of vortices crossing the boundary $|\rho|\nabla v \cdot n$ known.

\bigskip 

In fact, when $\rho \ge 0$ it becomes of our class \eqref{eq:ADE} (or \eqref{eq:ADE Omega} in bounded domains). 
On $\Rd$ we can solve as before to recover
$
	V_t = W * \rho_t
$
where $W = \mathrm{F^{-1}}[ (1 + |\xi|^2)^{-1} ]$.
As long as $\rho \ge 0$, this model corresponds to \eqref{eq:ADE} where $U, V = 0$. 
If we work in a bounded domain $\Omega$ the homogeneous Dirichlet condition we recover \eqref{eq:ADE Omega} where $U, V = 0$ and $K$ is the Green kernel of the equation. Notice that $K$ is not a compactly supported in $\Omega \times \Omega$.

\bigskip 

There is extensive literature on this problem. First, there were different works applying PDE theory. First,  
\cite{chapman1996MeanfieldModelSuperconducting} proved existence of the problem posed in bounded domains. Then in
\cite{huang1998EvolutionMixedStateRegions} the authors work with non-negative $\rho_0$ and compactly supported, using generalised characteristics.
We also point to 
\cite{schatzle1999AnalysisMeanField} where the authors perform a vanishing viscosity analysis in bounded domains.

Later, some authors applied Optimal Transport methods, in particular the 2-Wasserstein gradient-flow approach. 
In \cite{ambrosio2008GradientFlowApproacha} the authors worked in with non-negative probability measures in a bounded domain, i.e., $\mathcal P_2 (\overline \Omega)$.
Later, the work was extended in \cite{ambrosio2011GradientFlowChapman} to a bounded domain and signed measures. They extend the Wasserstein distance to signed measures and in fact they minimise a localised energy in the space of signed unit measures in $\Rd$ to recover a problem of the form
$$
	\frac{\partial \rho}{\partial t} = \diver( \mathbf 1_\Omega |\rho| \nabla V_t ), \qquad \text{for } t> 0, x \in \Rd.
$$
This sparked broader interest in the family of aggregation equations \eqref{eq:AE}.

\subsection{Newtonian vortex problem} 
In \cite{lin1999HydrodynamicLimitGinzburgLandau} the authors justify that a different limit from Gizburg-Landau equations leads to the problem without zero-order term
\begin{equation*}
	\tag{NVE}
	\label{eq:Newtonian vortex}
	\begin{dcases}
		\frac{\partial \rho}{\partial t} = \diver (\rho \nabla V_t ) \\ 
		-\Delta V_t = \rho.
	\end{dcases}
\end{equation*}
Again, this can be set in $\Rd$ or in a bounded domain (with suitable boundary conditions).

\bigskip 

This problem also attracted several authors, who brought different approaches.
\cite{lin1999HydrodynamicLimitGinzburgLandau} the authors construct solutions by characteristics with the so-called ``vortex-blob'' method.
In \cite{masmoudi2005GlobalSolutionsVortex} the authors prove existence of renormalised solutions. 
The authors of \cite{bertozzi2012AggregationNewtonianPotential}  study positive and negative solutions. This is equivalent to studying non-negative solutions of the problems
$$
	\frac{\partial \rho}{\partial t} = \pm \diver (\rho \nabla W_{\rm N} * \rho ). 
$$
The $W = -W_{\rm N}$ can also be seen of the problem $W = W_{\rm N}$ but taken backwards in time.
In fact, the authors prove that the $+$ sign leads to asymptotic dispersion whereas $-$ leads to blow-up.

\subsection{The Caffarelli-Vázquez problem}

In \cite{CaffarelliVazquez2011a}, Caffarelli and Vázquez introduced a non-local porous medium-type equation given by
\begin{equation*}
	\begin{dcases}
		\frac{\partial \rho}{\partial t} = \diver (\rho \nabla V_t ) \\ 
		(-\Delta)^s V_t = \rho ,
	\end{dcases}
\end{equation*}
where $(-\Delta)^s$ denotes the fractional Laplacian. In $\Rd$ this operator is to the operational fractional power of the Laplacian, which can be rigorously defined as the operator with Fourier symbol $|\xi|^{2s}$.
This problem is often called \textit{Caffarelli-Vázquez problem} and is often written in the compact form
\begin{equation*}
	\label{eq:Caffarelli-Vazquez}
	\tag{CVE}
	\frac{\partial \rho}{\partial t} = \diver (\rho \nabla (-\Delta)^{-s} \rho ).
\end{equation*}
This is a ``diffusive'' operator. To solve the fractional Laplace equation we take $W = \mathrm F^{-1}[|\xi|^{-2s}]$, and we recover the so-called \textit{Riesz potential}
$$
	W_{\rm R} (x) = c_{d,s} 
	\begin{dcases}
		 |x|^{2s-d} & \text{if } 2s < d , \\
		\log|x| & \text{if } 2s = d
	\end{dcases}
$$
The properties convolution properties of the Riesz potential where already well known in Harmonic Analysis.
There has been a lot of analysis for this equation: asymptotic analysis \cite{CaffarelliVazquez2011b}, 
\added[id=O]{%
sign-changing case and self-similar solutions \cite{biler2015NonlocalPorousMedium},
}%
regularity theory \cite{CaffarelliSoriaVazquez2013,CaffarelliVazquez2015}, and in \cite{SerfatyVazquez2014} that the limit as $s \to 1$ is \eqref{eq:Newtonian vortex}. 
In \cite{lisini2018GradientFlowApproach} the authors construct a solution through the Wasserstein gradient-flow theory described in \Cref{sec:optimal transport}.

\subsection{McKean-Vlasov and Kuramoto problems} 
\label{sec:McKean-Vlasov}

The McKean-Vlasov process is given by the stochastic differential equation
$$
	\diff X_t = a_t \diff \mathcal B_t + b_t \diff t
$$
where $a_t, b_t$ may depend on $X_t$ and its law, and $\mathcal B_t$ is a Wiener process, i.e., we introduce \emph{Brownian motion}.
We can extend this idea to a system of $N$ particles driven by coupled McKean-Vlasov equations. 
If we consider that $a_t$ is constant, say $a_t = \sqrt{2 \beta^{-1}}$, and $b_t$ is given by the interaction we can arrive at
$$
	\diff X_t^{(i)} = -V(X_t^{(i)}) -\frac{1}{N}\sum_{\substack{j=1 \\ j \ne i}}^N \nabla W(X_t^{(i)}-X_t^{(j)}) \diff t + \sqrt{2 \beta^{-1}} \diff \mathcal B_t.
$$
Through arguments of mean-field theory (or \emph{propagation of chaos} see, e.g., \cite{Jabin2017}), one arrives asymptotically as $N \to \infty$ to \eqref{eq:ADE} with linear diffusion $U = \beta \rho \log \rho$ (see, e.g., \cite{carrillo2020LongTimeBehaviourPhase}).
This is the reason the problem
\begin{equation*}
	\tag{McKVE}
	\label{eq:McKean-Vlasov}
	\frac{\partial \rho}{\partial t} = \beta^{-1} \Delta \rho + \diver( \rho \nabla (V + W* \rho)  )
\end{equation*}
is sometimes called \emph{McKean-Vlasov problem} (see, e.g., \cite{carrillo2020LongTimeBehaviourPhase}).
However, this denotation is not universal (see, e.g., \cite{Carrillo+GC+Yao+Zheng2023}).

\bigskip

Y. Kuramoto \cite{Kuramoto1975} introduced a model for synchronisation of chemical and biological oscillators, and it has been well-received by the community in neuroscience. 
It is a particular case of McKean-Vlasov problem in $d = 1$ where 
$X_t^{(i)}$ models the phase of a coupled oscillator. Therefore, we take $X_t^{(i)} \in [0,2\pi)$ or, equivalently, periodic boundary condition. 
The particular choices are $V(x) = -\omega_i$ the natural frequency of the system, and $W(x) = \sin(x)$. The resulting \eqref{eq:ADE} is known as the \emph{Kuramoto model}
\begin{equation*}
	\tag{KE}
	\label{eq:Kuramoto}
	\begin{dcases} 
		\frac{\partial \rho}{\partial t} = \beta^{-1} \Delta \rho + \diver( \rho (-\omega + \sin * \rho)  ) & t > 0, x \in (0,2\pi), \\
		\rho_t \text{ is }  2\pi - \text{periodic}.
	\end{dcases} 
\end{equation*} 
The case $\beta^{-1} = 0$ is also interesting. An interesting analysis of this problem without diffusion (i.e., $\beta^{-1} = 0$) can be found in \cite{carrillo2014ContractivityTransportDistances}.

\subsection{The neural networks of machine learning}

A neural network is no more than a specific type of parametric function accepting an input vector $x \in \Rd$ and returning $y \in \R$. If the output is higher dimensional, then we can create a neural network for each output.
In this context ``training'' corresponds to optimising the parameters to fit a set of prescribed data (supervised learning) or to achieve some objective (unsupervised learning). 
Neural networks are form by connecting so-called \emph{perceptron}. A perceptron is a parametric function of the form
$$
	x \longmapsto \sigma( \theta_{1} \cdot x + \theta_{2} )
$$
where $\sigma$ is called the activation function and $\theta_{1} \in \Rd , \theta_{2} \in \mathbb R$.
A \emph{one-layer neural network} is the linear combination of the output of $N$ perceptrons, and can be expressed
$$
	f_N( x | w, \theta) \defeq \sum_{i=1}^N \frac{w_i}{N} \sigma( \theta_{1,i} \cdot x + \theta_{2,i} )
$$
where $w_i \in \R$, $\theta_{1,i} \in \Rd$, and $\theta_{2,i} \in \mathbb R$.
The aim is, for $N$ fixed, to find the best parameter to approximate a target function $f$, in the sense that
\begin{equation} 
	\label{eq:ml minimisation}
	\min_{(w, \theta)} \int_\Rd | f(x) - f_N( x | w, \theta)  |^2\diff x .
\end{equation}
The quadratic error can be replaced by more general functions.

\bigskip

In \cite{fernandez-real2022ContinuousFormulationShallow} the authors give a nice presentation of the following \emph{continuous formulation}, i.e., limit as $N \to \infty$.
When can let $\xi_i \defeq (w_i,\theta_{1,i},\theta_{2,i})$, $\Xi(\xi,x) \defeq w \sigma(\theta_{1,i} x + \theta_{2,i})$, and define the empirical distribution
\begin{equation}
	\label{eq:empirical distribution NN}
	\mu_N \defeq \frac 1 N \sum_{i=1} \delta_{\xi_i}.
\end{equation} 
For convenience, define $\Omega$ as the set of admissible $\xi$.
Then, the value of the neural network can be computed as
$$
	f_N( x | w, \theta) = \int_\Omega \Xi(\xi,x) \diff \mu_N(\xi). 
$$
We can rewrite the \eqref{eq:ml minimisation} to a minimisation problem over the set of empirical distributions given by \eqref{eq:empirical distribution NN}. The limit as $N \to \infty$ is now clear, we pass to the minimisation problem over $\mathcal P (\Omega)$.
To smoothen the problem the authors allow for a penalisation $\widehat V$ (of type $|\xi|^2/2$), which can be understood as trying to prevent \textit{overfitting}. We land on the problem
$$
	\min_{\mu \in \mathcal P (\Omega)} \frac  1 2 \int_\Rd \left| f(x) - \int_\Omega \Xi(x, \xi) \diff \mu (\xi) \right|^2\diff x + \int_\Omega \widehat V(\xi) \diff \mu(\xi)  .
$$
Expanding the square, this is precisely \eqref{eq:free energy star} for $U = 0$ and
\begin{align*} 
	K(x,y) &= \frac 1 2 \int_\Omega \Xi(x,z) \Xi(y,z) \diff z, \\
	V(x) &= -\int_\Omega \Xi(x,y) f(y) \diff y + \widehat V(x).
\end{align*} 
To find local minimums of this energy, it is natural over $\mathcal P(\Omega)$, it is natural to consider the associated $2$-Wasserstein gradient flow. As will show below, this leads to \eqref{eq:ADE Omega}.

\bigskip

In \cite{chizat2022InfinitewidthLimitDeep} the authors generalise this result to problems with more than one layer.

\subsection{Granular flow equation}
\added[id=R1, comment=This is a nice example. We thank the referee for the suggestion]{%
Another popular example of the \eqref{eq:ADE} family is the so-called granular flow equation.
This phenomenom is usually modelled in phase space $(t,x,v)$ where $v$ is the velocity.
In that model
$\rho(t,x,v)$ represents the phase-space distribution.
We arrive at \eqref{eq:ADE} if the $\rho(0,x,v)$ does not depend on $x$. 
In this setting, $U$ models the random interactions of the granules with their environment (a fluid or heat bath), $V$ models  friction, and $W$ models inelastic collisions between granules with different velocities.
For more details see \cite{villani2006MathematicsGranularMaterials,benedetto1998NonMaxwellianSteadyDistribution,CarrilloMcCannVillani2006}.
}%

\subsection{The power-type family of Aggregation-Diffusion Equations}

Many authors have devoted their attention to the power-type family of non-linearities given by $U = U_m$, 
$$  
	V_\lambda (x) = 
    \begin{dcases} 
        \frac{|x|^\lambda}{\lambda} & \text{if } \lambda \ne 0 \\
        \log|x| & \text{if } \lambda = 0 ,
    \end{dcases}, \qquad \qquad 
	W_k (x) = 
    \begin{dcases} 
        \frac{|x|^k}{k} & \text{if } k \in (-d,d) \setminus \{0\} \\
        \log|x| & \text{if } k = 0 ,
    \end{dcases}
$$
Notice that
$$
	W_{\rm N} = - C_{d} W_{2-d} \quad \text{ and } \quad W_{\rm R} = - C_{d,s} W_{2s-d} \qquad \text{ with } C_{d} > 0, C_{d,s} > 0
$$
This family covers most of the relevant examples above.
We introduce the associated free energies
\begin{equation}
    \label{eq:free energy parameters} 
    \calU_m[\rho] = \int_\Rd U(\rho), \qquad \mathcal V_\lambda[\rho] = \int_\Rd V_\lambda \rho, 
	\qquad
	\mathcal W_k[\rho] = \frac 1 2 \iint_\Rd W_k(x-y) \rho(x) \rho(y) \diff x \diff y.
\end{equation}

\section{Well-posedness frameworks}

	After setting a model like the ones described in \Cref{sec:examples}, and before any other analysis can take place, the mathematician immediately aims to solve the question of existence and uniqueness. 
For this, one needs a suitable framework in which to define what is a ``solution''. 

\bigskip 

The very first approach to a PDE is to consider the existence of so-called \emph{classical} solutions. These are solutions that are smooth enough so that all the derivatives in the equation can be taken, and the equation is satisfied is \emph{all} points in the domain.

\bigskip

The hope for classical solutions quickly vanishes for many singular problems (e.g., \eqref{eq:PME} when $m > 1$). To deal with this difficulty, many authors in 20th century dedicated their efforts to the notion of weak solutions, with their derivatives taken in \emph{distributional sense}. So the equation could simply be satisfied in \emph{almost all points}, and additional reasonable conditions.
We will make some comments on this in \Cref{sec:PDE approach}. 

\bigskip 

We devote
\Cref{sec:a priori estimates} to present some good and bad expected properties in different setting. In particular, we will also make a brief stop in \Cref{sec:stationary solutions} to discuss stationary solutions (i.e., solutions no evolving in time).
When not even weak solutions may exist (like in the transport equation), the community turned towards \emph{Optimal Transport} techniques, which we will discuss in \Cref{sec:optimal transport}. 
    \subsection{The PDE approach}
\label{sec:PDE approach}

First, we discuss the more classical approach to understand ``solutions'' of \eqref{eq:ADE}.

\subsubsection{Local PDEs}
Non-linear diffusion problems of the general type
\begin{equation}
	\label{eq:PDE}
	\frac{\partial \rho}{\partial t} = \Delta \Phi(\rho) + \diver ( \mathbf E(t,x, \rho(t,x)))
\end{equation}
were widely studied in the 20-th century. Through the work of many authors, different types of solutions have been studied: classical, weak, entropy, viscosity, …

\bigskip

Let us start by discussing the nicer problem given by non-degenerate parabolic problems.
When $\Phi, E$ are smooth and we assume $\Phi$ is uniformly elliptic, in the sense that there exist constants such that
\begin{equation}
\label{eq:Phi elliptic}
	0 < c_1 \le \Phi'(u) \le c_2 < \infty,
\end{equation}
existence, uniqueness, and maximum principle hold from the classical theory.
The literature is extensive:
in $\Rd$ this issue was solved at the beginning of the twentieth century (see \cite{ladyzhenskaia1968parabolic}). This book already contains several results for bounded domains (see \cite[Chapter V.6 and V.7]{ladyzhenskaia1968parabolic}). 
However, the development of regular solutions would leave room for later improvement. A good reference for Dirichlet boundary conditions is \cite{DiBenedetto1993b}. There is a series of works by \cite{Amann1990} for Neumann/Robin boundary conditions, although they provide very general ellipticity conditions which written for systems and are not easy to apply to specific case.
There are later references which are more concise, e.g., \cite{Yin2005}. 

\bigskip

The theory of weak solution is based on the notion of weak (or distributional) derivative, and hence solutions are found in Sobolev spaces. 
This is typical of parabolic problems where $\Phi \ne 0$. There are more general theory that allow for well-posedness even when $\Phi = 0$: entropy solutions (see \cite{Carrillo1999,KarlsenRisebro2003} and the references therein) and viscosity solutions (e.g., \cite{CrandallIshiiLions1992} and the references therein)

\bigskip 

The theory of weak and energy solutions for \eqref{eq:PME} is well presented in the book \cite{Vazquez2007}. Notice that for $m > 1$ we have $\Phi'(0) = 0$, and we say that the diffusion is degenerate, and for $m < 1$ we have $\Phi'(0) = \infty$, and we say the diffusion is singular. The first order problems $U = 0$ (or, equivalently, $\Phi = 0$) are purely of the first order and hence may yield discontinuous solutions.

\begin{remark}
	Occasionally, Hölder spaces are used in addition to Sobolev. For this, one can use the method of intrinsic scaling (see \cite{DiBenedetto1993b,urbano2008MethodIntrinsicScaling} and the references therein)
\end{remark}

\begin{remark}
	\label{rem:Crandall-Liggett}
	Purely-diffusive problems like \eqref{eq:PME} admit semigroup-type solutions. In the linear theory this is called Hille-Yosida theorem (see, e.g., \cite[Chapter 7]{Brezis2010}) and in the non-linear setting Crandall-Liggett theorem (see \cite{CrandallLiggett1971}). 
	The semigroup is constructed as the limit $\tau \to 0$
	of solutions of the implicit Euler scheme $u_0 = \rho$ and
	$$
		u_{k+1} + \tau (-\Delta u_{k+1}^m) = u_k. 
	$$
	To be precise, we take $t_k = k \tau$ and for $t \in [t_k, t_{k+1})$ then we define
	\begin{equation*}
		\rho^{(\tau)}_t = (1-\tfrac{t-t_k}{\tau}) u_k + \tfrac{t-t_k}{\tau} u_{k+1} .
	\end{equation*}
	
\end{remark} 

\begin{remark}
	Boundary conditions of homogeneous Dirichlet (i.e., $u = 0$) and Neumann (i.e., $\nabla u \cdot n = 0$) are fairly well understood in terms of existence and uniqueness (see, e.g., \cite{ladyzhenskaia1968parabolic}). 
	In linear problems, even so-called Robin type boundary conditions $\nabla u \cdot n + \alpha u= 0$.
	The operator $Au = -\Delta u - \diver(u \nabla V)$ with no-flux condition is monotone (i.e., $\langle Au, u \rangle \ge 0$) if $\| \nabla V \|_{L^\infty}$ is small enough, so well-posedness follows from the Hille-Yosida theory.
	Non-linear diffusion problems are more difficult.
	Works for \eqref{eq:Keller-Segel un-normalised} and related problem usually take advantage of an auxiliary boundary condition for $V_t$.

	When there is no diffusion (i.e., $U, W = 0$) and $V = -|x|^2$ leads to mass trying to exit any ball $B_r$. If work on a ball $B_R$ and we set no-flux condition $\rho \nabla V \cdot n = 0$ on $\partial B_R$, the result will be the formation of Dirac deltas on $\partial B_R$. This can be properly understood by the optimal transport theory below. It points to a subtle trade-of between diffusion and aggregation.

	Existence and uniqueness for \eqref{eq:ADE Omega} is only understood with \eqref{eq:BC Omega star} if we assume \eqref{eq:V and K boundary condition} (see \cite{Carrillo+Fernandez-Jimenez+GC2023} for $U = U_m$ with $m \in (0,1)$). The general case when \eqref{eq:V and K boundary condition} does not hold seems to be open. We provide a tricky example with no-diffusion in \Cref{eq:transport in a bounded domain}.
\end{remark} 

\subsubsection{Weak vs very weak solution}

If we work in $\Rd$ we can take $\varphi \in C_c^\infty ((0,T) \times \Rd)$, multiply the equation and integrate one to recover
$$
	-\int_0^\infty \int_\Rd \rho_t \frac{\partial \varphi}{\partial t} = - \int_0^\infty \int_\Rd \rho \nabla( U'(\rho) + V + W*\rho) \cdot \nabla \varphi
$$	
Functions $\rho$ that satisfy this equation are so-called \emph{weak solutions} of the problem.
Here, $U(\rho)$ must have a gradient in some sense. 
In $\Rd$ we can integrate by parts again in the diffusion term to recover
$$
	-\int_0^\infty \int_\Rd \rho_t \frac{\partial \varphi}{\partial t} = \int_0^\infty \int_\Rd \Phi(\rho) \Delta \varphi - \int_0^\infty \int_\Rd \rho \nabla( V + W*\rho) \cdot \nabla \varphi. 
$$	
These are the so-called \emph{very weak} solutions to the problem.

\bigskip 

If we work in a bounded domain $\Omega$ and consider the no-flux condition we can recover weak solutions. If we want to recover the no-flux condition from this weak formulation, we need to take $\varphi$ not vanishing in the boundary, for example $C_c^\infty((0,T); C^\infty (\overline \Omega))$.
To integrate by parts again we would need that $\partial_n \Phi (\rho) = 0$ on the boundary. This is only possible if we assume that $V$ and $K$ do not interfere with the boundary \eqref{eq:V and K boundary condition}.

\subsubsection{Solutions of \texorpdfstring{\eqref{eq:ADE Omega}}{(ADE*)} by approximation} 

We will not make any detailed analysis of Sobolev spaces let us make some comments that the equation would have if the solutions where smooth. 
We define $\Phi'(s) = s U''(s)$ and $\Phi(0) = 0$. 
To give a hint on the procedure:
\begin{enumerate}
	\item For uniformly elliptic $\Phi_\ee$, e.g.,
	$$\Phi'_\ee \sim \ee + (\Phi' \wedge \ee^{-1}) $$ 
	and $E_t(x)$ known, existence is done through the classical theory above.

	\item Letting $\ee \to 0$ yields solutions of the degenerate/singular diffusion term. This is the approach of \cite{Vazquez2007}.

	\item Take $V_\kappa, K_\kappa$ smooth approximations of $V$ and $K$. Then replace $E_t(x)$ by $\rho \nabla (V + \calK \rho)$ by a fixed-point argument.

	\item Pass to the limit in $\kappa$.
\end{enumerate}
The order of steps 2 and 3 can be suitable exchanged depending on convenience. They depend on delicate a priori estimates.
For two particular examples for global-in-time existence and uniqueness theory using this approach we point to \cite{CarrilloHittmeirVolzoneYao2019} for  $m > 1$, \cite{Carrillo+GC+Yao+Zheng2023} for $m = 1$, and \cite{Carrillo+Fernandez-Jimenez+GC2023} for $m \in (0,1)$ (the arguments in \cite{Vazquez2007} make $m > 1$ simpler).
For \eqref{eq:AE} we point the reader to \cite{lagoutiereVanishingViscosityLimit} for a vanishing viscosity argument.

\subsubsection{Push-forward solutions}
\label{sec:examples transport}

When the velocity field $v$ is known, the transport can be written as
\begin{equation*}
	\frac{\partial \rho}{\partial t} + \diver(\rho v_t ) = 0.
\end{equation*}
If $v$ is smooth problems, then can be solved by characteristics. 
This problem can also be recovered trying to find solutions as generalised characteristics  $\rho_t(X_t(y)) = A_t \rho_0(y)$.
This leads the set of ODEs
\begin{equation}
	\label{eq:transport characteristics}
    \begin{cases}
        \dfrac{\partial X_t}{\partial t} = v_t(X_t) & t > 0, \\
        X_0(y) = y.
    \end{cases} 
\end{equation}
If $v_t$ is Lipschitz in $x$, the field of characteristics is the unique solution of this problem. We do not concern ourselves with $A_t$ for now.

\bigskip 

A more elegant approach is to look for solutions of the Lagrangian formulation of the problem.
\begin{equation}
	\label{eq:lagragian formulation}
	 \int_{X_t(A)} \rho_t(x) \diff x =  \int_{A} \rho_0(x) \diff x.
\end{equation} 
This is equivalent to
\begin{equation*}
	\int_{\Rd} \rho_t(x) \varphi(x) \diff x = \int_\Rd \rho_0(x) \varphi(X_t(x)) \diff x.
\end{equation*}
For the details we recommend \cite[Lecture 16]{AmbrosioBrueSemola2021}.
This taps into a well-understood theory coming from the \textit{Optimal Transport} world (we will come back to this in \Cref{sec:optimal transport}).
Given a measure $\mu \in \mathcal P(\Omega)$ and a measurable map $T: \Omega \to \Omega$ we can define the push-forward of $\mu$ by $T$, which we denote $T_\sharp \mu \in \mathcal P(\Omega)$, as
$$
    T_\sharp \mu (B) \defeq \mu(T^{-1}(B)) \qquad \text{ for all } B \text{ measurable.}
$$
In particular, we can rewrite as \eqref{eq:lagragian formulation}
$$
	\rho_t = (X_t)_\sharp \rho_0.
$$
There are many properties of $\rho_t$ that be derived directly from this structure and the analysis of the ODEs \eqref{eq:transport characteristics}.

\subsubsection{Solutions by characteristics when \texorpdfstring{$U = 0$}{U = 0}}

If $U, W = 0$, then we can simply solve the de-coupled PDEs
$$
\begin{cases}
	\dfrac{\partial X_t}{\partial t} = -\nabla V(X_t) & t > 0, \\
	X_0(y) = y,
\end{cases} 
$$
and set $\rho_t = (X_t)_\# \rho_0$.
When $V$ is of class $C^2$, then a unique solution holds by the Picard-Lindelöf theorem. Furthermore, coming back to \Cref{sec:transport with known potential} if we have $\nabla V(x) \cdot x \ge 0$ then
$$
	\frac{\diff}{\diff t} | X_t | ^2 = - \nabla V(X_t) \cdot X_t \le 0
$$
so all the mass is moving towards $0$. 
Since $X_t = 0$ is a solution, when $V \in C^2$ mass will concentrate in infinite time.
A posteriori, for \eqref{eq:ADE} when $U = 0$ one can always define
$$
	v_t = \nabla (V + W*\rho)
$$
and arrive at a transport problem.

\bigskip

The problem \eqref{eq:AE} is more involved. The specific case \eqref{eq:CRSE} was done \cite{huang1998EvolutionMixedStateRegions}, where the author arrive at closed-form formulas for the characteristics.
The author of \cite{laurent2007LocalGlobalExistence} proves global existence of solutions when $W$ is good enough by a fixed-point argument with suitable a priori estimates in Sobolev spaces.
This result later applied to construct local-in-time solutions by approximation 
\added[id=R1]{%
	\cite{BertozziCarrilloLaurent2009,bertozzi2010FinitetimeBlowupWeaka,bertozzi2011LpTheoryMultidimensional}, 
}%
where a key difficulty is to then decide if there is global existence or finite-time blow-up.
Later works use characteristics to study blow-up behaviours (see, e.g., \cite{bertozzi2012CharacterizationRadiallySymmetrica%
}).
		\subsubsection{Stationary solutions}
\label{sec:stationary solutions}

There is a particular type of solutions which will be of interest for the asymptotics.
A \emph{stationary solution} is a solution that does not evolve in time. 
We recall the equivalent writing \eqref{eq:Wasserstein grad flow}.
If $\widehat \rho$ is a stationary solution, using $\varphi = \frac{\delta \calF}{\delta \rho} [\widehat \rho]$ as a test function yields
$$
    \int_\Omega \widehat \rho \left|\nabla \frac{\delta \calF}{\delta \rho} [\widehat \rho] \right|^2 = 0.
$$
Hence, we arrive at the equation for stationary states
\begin{equation*}
    \tag{ADE-S}
    \label{eq:ADE stationary state}
    \widehat \rho \, \nabla \frac{\delta \calF}{\delta \rho} [\widehat \rho] = 0, \qquad \text{a.e. in } \Omega.
\end{equation*}
This means that in open connected sets where $\hat \rho > 0$ then $\frac{\delta \calF}{\delta \rho} [\widehat \rho]$ is constant. 

Usually it is asked also that $\rho \in L^1 \cap L^\infty$ that $\Phi(\rho) \in W^{1,1}_{loc}$ (where $\Phi'(s) = s U'(s)$ and $\Phi(0) = 0$) and that $\nabla V, \nabla W * \rho \in L^1_{loc}$. This regularity is not always achievable.

\begin{remark}
    When they are available, these kinds of solutions will typically be recovered from a minimisation problem. We will discuss particular examples in \Cref{sec:existence of minimisers}.
\end{remark}

\begin{remark}
    Notice that solving this problem when $U = U_m$ and $V = V_2$ the Barenblatt profile \eqref{eq:Barenblatt profile} (i.e., \eqref{eq:Barenblatt} in re-scaled variables) is a solution of \eqref{eq:ADE stationary state}.
\end{remark}

\subsection{A priori estimates for smooth solutions}
\label{sec:a priori estimates}

In this section we make some formal computations to understand properties of the solution of \eqref{eq:ADE Omega}. We will mostly assume that $\rho$ is smooth. However, due to weak lower semi-continuity of norms, the a priori estimates still usually hold after approximation.

\subsubsection{Conservation of mass}
\label{sec:conservation of mass}
Let us consider \eqref{eq:ADE Omega}, and take a smooth set $\omega \subset \Rd$. We can formally compute the balance of mass 
$$
	\frac{\diff }{\diff t} \int_\omega  \rho \diff x = 	 \int_\omega \frac{\partial \rho}{\partial t} = \int_\omega \diver \left( \rho \nabla (U'(\rho) + V + \mathcal K \rho) \right)\diff x = -\int_{\partial \omega} \rho \nabla (U'(\rho) + V + \mathcal K \rho) \cdot n \diff S.
$$
Hence, if we have the boundary conditions \eqref{eq:BC Omega} we have conservation of mass. 
In $\Rd$ we will need that the decay as $|x| \to \infty$ is fast enough.
This is not always the case, as shown by the next result (see \cite{BenilanCrandall1981}).

\begin{proposition}
	\label{prop:PME finite-time extinction}
	Let $\rho$ be a solution of \eqref{eq:PME} in $\Rd$ with $m \in (0, \frac{d-2}{d})$ and $d \ge 3$. Then
	\begin{enumerate}
		\item If $\rho_0 \in L^{q} (\Rd)$ with $q = \frac{(1-m)d}{2}$ then there is finite-time extinction.
		\item If $\rho_0 \in L^1 (\Rd)$ then there is (at least), infinite-time extinction.
	\end{enumerate}
\end{proposition}
\begin{proof}  First let $u_0 \in L^q$. Then, we compute for sufficiently good solutions that
\begin{align}
	\label{eq:BenilanCrandall finite time extinction}
	\frac{\diff}{\diff t} \frac 1 q \int_\Rd \rho^q 
		&= - C \int_\Rd  | \nabla \rho^{\frac{m+q} 2 }|^2 
		\le 
		-C { \left( \int_\Rd \rho^{\frac{m+q}{2}2^*} \right)^{\frac{1}{2^*}} }
\end{align} 
using Sobolev's inequality
where $2^* = \frac 1 2 - \frac 1 d$. 
The equation  $\frac{\diff }{\diff t} X =  - C X^{\alpha}$ where $\alpha < 1$ has finite-time extinction in the sense $X_T = 0$. Thus, we have the finite-time extinction
\begin{equation*}
	\int_\Rd \rho^q \to 0, \qquad t \nearrow T.
\end{equation*}

Let $\rho_0 \in L^1(\Rd)$. Take $\ee > 0$ and an approximating sequence $\rho^{(\ee)}_0 \in L^q$ with $\| \rho_0 - \rho_0^{(\ee)} \|_{L^1} \le \ee$. For $t \ge T_\ee^*$
\begin{align*}
	\|\rho(t) \|_{L^1} 
		&
		\le \| \rho(t) - \rho^{(\ee)}(t) \| _{L^1}  + \|\rho^{(\ee)} (t) \|_{L^1}  \\ 
		&\le \| \rho_0 - \rho^{(\ee)}_0 \|_{L^1} + 0 \\
		&\le \ee. \qedhere
\end{align*} 
\end{proof} 

This effect is known as ``loss of mass through infinity'', and is typical to the \textit{very fast} diffusion. In some settings, this can be offset by the aggregation term. To prevent this loss of mass it suffices that $\rho_t$ is a \textit{tight} family (a notion we discuss in the next section).

\subsubsection{Mass escaping through infinity: tightness}

In the previous section we have discussed the problem of mass escaping through infinity ``as time passes''. This phenomenon is also related to a difficult technical problem when working on $\Prob (\Rd)$. We recall some classical results for working with probability measures.

First, let $\rho_n$ be a sequence of functions in $\Prob(\Rd) \cap L^1 (\Rd)$. If they have an a.e. limit we know, at most, that 
$$
	\int_\Rd \hat \rho \le \liminf_n \int_\Rd \rho_n = 1.
$$
But the equality needs to hold.
Furthermore, let us take a sequence in $\mu_n \in \Prob(\Rd)$. Since they are uniformly bounded measures, and $\mathcal M (\Rd)$ is the dual of $C_b(\Rd)$ by the Banach-Alaoglu theorem we have that the sequence $\mu_n$ is weak-$\star$ pre-compact, and take $\widehat \mu$ an accumulation point. It needs not lie in $\Prob(\Rd)$.

To ensure that the limit is a measure we need to add some information. Prokhorov's theorem (see, e.g., \cite[Theorem 5.1.3]{AmbrosioGigliSavare2005}) states that a family of measures is pre-compact in $\Prob (\Rd)$ if and only it is \emph{tight}.
\begin{definition}
	A family of probability measures $(\mu_a)_{a \in A}$ is tight if, 
	for every $\ee > 0$ there exists $K_\ee \subset \Rd $ compact such that 
	$$
		\mu_a (\Rd \setminus K_\ee) < \ee, \qquad \forall a \in A. 
	$$
\end{definition}

This means, informally speaking, that the \emph{tails} of $\mu_n$ hold ``uniformly small'' mass. 
\begin{remark}
	\label{eq:tightness by test function}
	To show that a family of probability measures $(\mu_a)_{a\in A}$ is tight it suffices to show that it is uniformly integrable against a suitable function.
	Following \cite[Remark 5.1.5]{AmbrosioGigliSavare2005}, it suffices to
	find $\varphi:\Rd \to [0,\infty]$
	such that its sublevel sets (i.e., $\{ x \in \Rd : \varphi (x) \le c  \}$ for $c \in \mathbb R$) are compact, and we have
	\begin{equation}
		\label{eq:tightness integral condition} 
		\sup_{a \in A} \int_\Rd \varphi(x) \diff \mu_a(x) < \infty .
	\end{equation} 
\end{remark}

Very typical examples is $\varphi (x) = |x|^p$ with $p \ge 1$. 
The quantity $\int_\Rd |x|^p \diff \mu$ is called the $p$-th order moment of $\mu$.

\bigskip 

It is clear that \eqref{eq:tightness integral condition} can follow for minimisers \eqref{eq:free energy} if $V$ is good enough. We point out the following estimate that allows to use $\calW$ for the same purpose.
\begin{lemma}[\cite{McCann1997,CarrilloCastorinaVolzone2015}]
	\label{lem:W controls tails}
	Let $\rho$ be radially symmetric, $\Omega = \Rd$, $w$ non-decreasing such that $w(r_0) \ge 0$ for some $r_0$, and let $r_1 > r_0$. Then
	$$
		\int_{|x| \ge r_1} w(|x|) \rho(x) dx \le  \frac{ \iint_{\Rd \times \Rd} | w(|x-y|) | \rho(x) \rho(y) \diff x \diff y  }{ \frac 1 2 \int_{\Rd} \rho(x) \diff x }   
	$$
\end{lemma} 
\begin{proof}
	Taking into account the geometry shown in \Cref{fig:W controls tails} we can estimate 
	\begin{align*}
		\int_{|x| \ge r_1} &\int_{|x-y| \ge r_0} w(|x-y|) \rho(x) \rho(y) \diff y \diff x  \\
		&\ge \int_{|x| \ge r_1} \int_{y \cdot x \le 0} w(|x-y|) \rho(x) \rho(y) \diff y \diff x 
	\end{align*} 
	\begin{figure}[H]
		\centering
		\includegraphics[width=.5\textwidth]{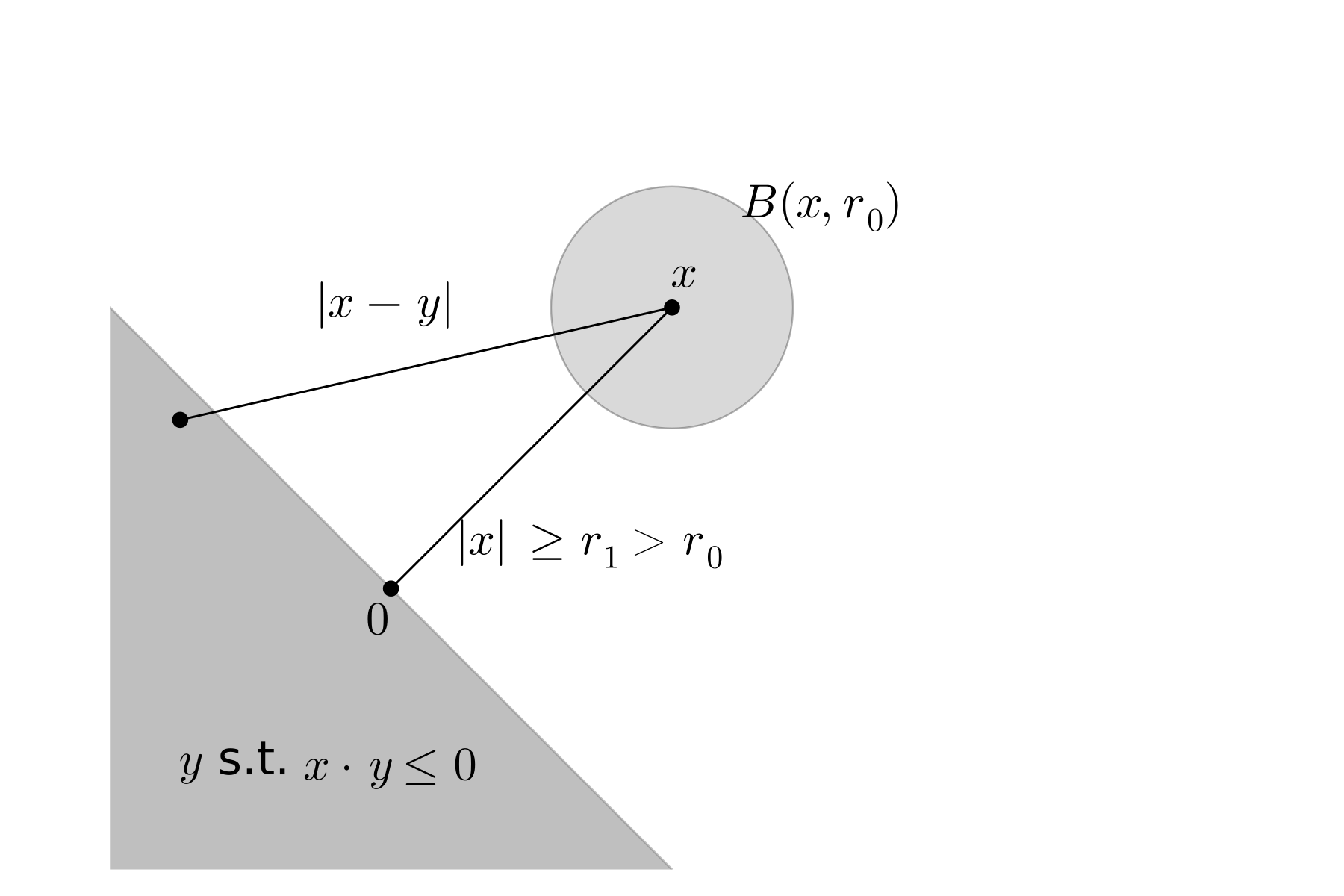}
		\caption{Sets used of the proof of \Cref{lem:W controls tails}}
		\label{fig:W controls tails}
	\end{figure} 
	Since $x \cdot y \le 0$ then $|x-y|^2 = |x|^2 + |y|^2 - 2 x \cdot y \ge |x|^2$ and since $w$ is non-decreasing
	\begin{align*} 
		&\ge \int_{|x| \ge r_1} \int_{y \cdot x \le 0} w(|x|) \rho(x) \rho(y) \diff y \diff x  \\
		&\ge \int_{|x| \ge r_1} w(|x|) \rho(x) \int_{y \cdot x \le 0}  \rho(y) \diff y \diff x.
	\end{align*}
	Since $\rho$ is radially symmetric, for any $x \in \Rd$ we have 
	\begin{equation*}
		\int_{y \cdot x \le 0}  \rho(y) \diff y = \frac 1 2 \int_{\Rd}  \rho(y) \diff y. \qedhere
	\end{equation*}
\end{proof}

\subsubsection{An \texorpdfstring{$L^p$}{Lp} estimate}
\label{sec:Lp estimate}

We can proceed similarly to the approach in \eqref{eq:BenilanCrandall finite time extinction} in more general settings.
If we drop the term coming from the diffusion, then we get
\begin{align*}
	\frac{\diff}{\diff t} \frac{1}{p} \int_\omega \rho^p &= \int_{\omega} \rho^{p-1} \frac{\partial \rho}{\partial t} \\
	&= \int_{\omega} \rho^{p-1} \diver \rho \nabla \frac{\delta \calF}{\delta \rho}\\
	&= -(p-1)\int_{\omega} \rho^{p-1} \nabla \rho \nabla (U'(\rho) + V + \calK \rho ) + \int_{\partial \omega } \rho^p \nabla \frac{\delta \calF}{\delta \rho} \cdot n \\
	&\le - \frac{p-1}{p} \int_\omega \nabla {\rho_t^{p}} \nabla (V+ \calK \rho) + \int_{\partial \omega } \rho^p \nabla \frac{\delta \calF}{\delta \rho} \cdot n 
\end{align*}
Integrating by parts one final time leads to
\begin{equation}
	\frac{\diff}{\diff t} \frac{1}{p} \int_\omega \rho^p \le \frac{p-1}p \int_\omega \rho^p \Delta (V+\calK\rho) - \frac{p-1}p \int_{\partial \omega} \rho^p \nabla (V+\calK\rho) \cdot n(x) +  \int_{\partial \omega } \rho^p \nabla \frac{\delta \calF}{\delta \rho} \cdot n.
\end{equation}
If we work in $\Rd$ with $V, W \in W^{2,\infty}$, we assume good decay of $\rho^p \nabla (V+\calK\rho)$ and $\rho^p \nabla \frac{\delta \calF}{\delta \rho}$ then the boundary term vanish and we can write an inequality for equation of the form
$$
	\| \rho_t \|_{L^p(\Rd)} \le \| \rho_0 \|_{L^p(\Rd)} e^{Ct}.
$$

However, when we work in $\Omega$ with \eqref{eq:BC Omega star} and we take $\omega = \Omega$, then the last term vanishes. However, we cannot relate the boundary term $\rho^p \nabla (V+\calK\rho) \cdot n(x) $ to the $L^p$ norm of $\rho$. In some cases it can be compensated by the negative diffusion term, but not in general. This is one of the reasons we assume \eqref{eq:V and K boundary condition} on the boundary.

\subsubsection{Free-energy dissipation}
\label{sec:free energy dissipation}

\begin{definition}[First variation]
	Let $X$ be a normed space, $\calF : X \to \mathbb R$ and fix $\rho_0 \in X$. We say $\calF$ it is Gateaux differentiable at $X$ if for all $\varphi \in X$ there exists  
	$$
		\diff \calF [\rho_0, \varphi] \defeq \lim_{\ee \to 0} \frac{\calF [\rho_0 + \ee \varphi] - \calF[\rho_0]}{\ee}.
	$$
	The function $\diff \calF$ is usually called \emph{first variation of $\calF$}. 
	Some authors extend this definition and require just that the limit exists for $\varphi$ in a dense subset of $X$.
	Functionals that admit first variation are called \emph{Gateaux differentiable}.
\end{definition}
When $X \subset L^1 (\Omega)$, the first variation can sometimes be represented by a function $f_0 \in L^1(\Omega)$ the sense that
$$
	\int_\Omega f_0(x) \varphi(x) \diff x = \diff \calF [\rho_0, \varphi], \qquad \forall \varphi \in L^\infty(\Omega).
$$
If this is the case, we denote $\frac{\delta F}{\delta \rho}[\rho_0] \defeq f_0$.
It is easy to check formally that, when $\mathcal F$ is given by \eqref{eq:free energy star} and $K(x,y) = K(y,x)$, then \eqref{eq:first variation}.
In many of the examples the first variation can be rigorously computed. 

\bigskip

If $K(x,y) = K(y,x)$ then the solutions of \eqref{eq:ADE Omega}
admits \eqref{eq:free energy star} as Lyapunov functional.
This functional is usually called \textit{free energy}. For smooth solutions of \eqref{eq:ADE Omega} (with sufficiently good decay as $|x| \to \infty$ in unbounded domains) we can compute
\begin{equation}
	\label{eq:free energy dissipation}
	\frac{\diff}{\diff t} \mathcal F[\rho] = - \int_\Omega \rho \left|  \nabla \frac{\delta \mathcal F}{\delta \rho}    \right|^2 + \int_{\partial \Omega}  \rho   \nabla \frac{\delta \mathcal F}{\delta \rho} \cdot n .
\end{equation}
When $\Omega = \Rd$ (with sufficiently good decay of $\rho \nabla \frac{\delta \mathcal F}{\delta \rho}[\rho_t]$) or we have the no-flux condition \eqref{eq:BC Omega star}, we recover the decay of the free energy, i.e.,
along solutions of \eqref{eq:ADE Omega} we have decay of  free-energy dissipation
\begin{equation}
	\label{eq:free-energy dissipation}
	\mathcal F[\rho_t] - \mathcal F[\rho_s] = -\int_s^t \int_\Omega \rho \left| \nabla \frac{\delta \mathcal F}{\delta \rho} [\rho] \right|^2 .
\end{equation}
In particular if $s < t$ then $\calF [\rho_s] > \calF [\rho_t]$.

\begin{remark}
	Often the free energy dissipation allows us to prove tightness.
	One approach is to recall \Cref{eq:tightness by test function} and notice that free energy dissipation includes a term $\int_\Omega V \rho$. 
	The term $\iint W(x-y) \rho(x) \rho(y)$ can also be useful due to \Cref{lem:W controls tails}.
\end{remark}

\subsubsection{Negative Sobolev spaces}

Given a set $\Omega$, the negative Sobolev spaces can be defined by distributions with finite norm
$$
	\| f \|_{W^{-1,1} (\Omega)} =  \inf_{   f = \diver F  }	\|F \|_{L^1 (\Omega)}
$$
Using the free-energy dissipation
\begin{align*} 
	\Big \|\frac{\partial \rho_t}{\partial t} &  \Big\|_{W^{-1,1} (\Omega)} =  \inf_{ \frac{\partial \rho_t}{\partial t} = \diver F     }	\|F \|_{L^1 (\Omega)}
	\le \|\rho \nabla \tfrac{\delta \mathcal  F}{\delta \rho} \|_{L^1} \\
	&\le \|\rho_t\|_{L^1}^{\frac 1 2} \left(   \int_\Omega \rho_t \left|  \nabla \tfrac{\delta \mathcal F}{\delta \rho}  \right|^2 \right)^{\frac 1 2} 
		= \|\rho_t\|_{L^1}^{\frac 1 2} \left( -  \frac{d}{dt} \mathcal F[\rho_t] \right)^{\frac 1 2}
\end{align*} 
Also $\|\rho_t \|_{L^1} \le \| \rho_0 \|_{L^1}$. 
In the smooth setting, we can integrate in time
$$
	\|\rho_t - \rho_s \|_{W^{-1,1}} \le C |t-s|^{\frac 1 2}  (\mathcal F[\rho_s] - \mathcal F[\rho_t])^{\frac 1 2}.
$$
Since $L^1$ is compactly embedded in $W^{-1,1}$, when $\calF$ is bounded below, the sequence
$$
	\rho^{[n]}_t = \rho_{n+t}
$$
is relatively compact in $C([0,1]; W^{-1,1}(B_R))$ due to Ascoli-Arzelá theorem. 
Let $\widehat \rho \in W^{-1,1}$ be the limit of a sub-sequence. Since $\mathcal F[\rho_t]$ is bounded below and non-increasing, there is a limit $\widehat \calF$. Then
$$
	\| \widehat \rho_t - \widehat \rho_s \|_{W^{-1,1}} \le C|t-s|^{\frac 1 2} (\widehat \calF - \widehat \calF) = 0.
$$
So $\widehat \mu$ does not depend on $t$. In the cases where one can prove stronger convergences, it can be shown that it is a stationary solution, i.e., 
\begin{equation} 
	\label{eq:ADE stationary eq}
	\diver\left(\widehat \rho \nabla \frac{\delta \mathcal F}{\delta \rho}[\widehat \rho]\right) = 0.
\end{equation}

    \subsection{Optimal-transport approach: Wasserstein spaces and the JKO scheme}
\label{sec:optimal transport}

As we already mentioned in \Cref{sec:examples transport}, there is a clear connection between some of our equations and \emph{Optimal Transport} problems.
Otto realised in \cite{Otto1996} (see also \cite{Otto2001}) that the \eqref{eq:PME} in $\Rd$ corresponds, in fact, to the gradient flow of a $\mathcal F$ with respect to the 2-Wasserstein distance. This idea led to the so-called ``Otto calculus'' that was later perfected.

\bigskip

The following pages are meant as an extremely brief introduction. For the reader interested in deepening their understanding of these spaces and their connection to PDEs we suggest reading 
the lecture notes \cite{AmbrosioBrueSemola2021} first, and then the very detailed books \cite{AmbrosioGigliSavare2005,Villani2009}.
A very nice presentation with an emphasis on the examples can be found in \cite{Santambrogio2015}.

\subsubsection{The Wasserstein distances}

The Wasserstein metrics were a tool already used by people studying optimal transport between probability measures. 
We give a brief definition. Let $X$ be a metric space with a Borel algebra, and recall the definition of push-forward given in \Cref{sec:examples transport}.
If we are given $\mu, \nu \in \mathcal P(X)$, a transport map is a measurable function $T$ such that $\nu = T_\sharp \mu$. 
We can formally think about the $p$-Wasserstein distance between probability distributions are defined for $p \ge 1$ as
\begin{equation} 
    \label{eq:Monge problem}
    \left( \inf_{T:\nu = T_\sharp \mu} \int_{X} |x - T(x)|^p \diff \mu(x) \right)^{\frac 1 p}
\end{equation} 
This is the so-called Monge optimisation problem. The infimum is not always achieved, and sometimes no valid $T$ exists. Kantorovich improved this idea by introducing transport plans.
A transport plan is a probability distribution in $\mathcal P(X \times X)$ that have $\mu, \nu$, i.e.,
$$
    \mu(A) = \pi(A \times X) , \qquad \nu(A)= \pi(X \times A), \qquad \text{ for all } A \subset X \text{ measurable}.
$$
This set is denoted by $\Pi(\mu,\nu)$, and is never empty because of the measure $\mu \otimes \nu \in \Pi(\mu,\nu)$, the unique measure such that
$$
    (\mu \otimes \nu) (A \times B) = \mu(A) \nu (B).
$$
The \emph{Wasserstein distances} are defined as
\begin{equation}
    \label{eq:Wasserstein}
    \Wass_p(\mu,\nu) \defeq \left( \min_{\pi \in \Pi(\mu, \nu)}  \iint_{X \times X} |x-y|^p \diff \pi (x,y)   \right)^{\frac 1 p} .
\end{equation}
These distances are sometimes called also \emph{Kantorovich-Rubinstein distance}. 
The distance between two measures in $\mathcal P (X)$ can be infinite, unless the finite $p$-th order moment is finite. Hence, we define the $p$-Wasserstein space
$$
    \mathcal P_p (X) \defeq \left\{ \mu \in \mathcal P(X) : \int_{X} |x|^p \diff \mu (x) < \infty   \right\}.
$$
The pair $(\mathcal P_p (X), \Wass_p)$ is a metric space. If $X$ is complete, then so this is space is also complete and it is equivalent to the narrow convergence (see \cite[Proposition 7.1.5]{AmbrosioGigliSavare2005}). This is why we pick $X = \overline \Omega$. 

As it usually happens in these families of spaces, there are three highlighted cases: $p = 1,2, \infty$. 
The case $p = 2$ will appear in the next section, and it is the one directly related to \eqref{eq:ADE}.
The advantage of $p = 1$ is the so-called Kantorovich-Rubinstein duality
\begin{equation}
    \label{eq:Kantorovich-Rubinstein}
    \Wass_1 (\mu, \nu) = \sup \left\{  \int \psi \diff \mu - \int \psi \diff \nu : \text{for } \psi \text{ such that } \sup_{x \ne y}\frac{|\psi(x) - \psi(y)|}{|x-y|} \le 1    \right\}.
\end{equation}
A very interesting property of the Wasserstein distance is that it admits a dynamic characterisation through the so-called \emph{Benamou-Brenier formula} which is stated easiest in $\Rd$
\begin{equation}
    \label{eq:Benamou-Brenier}
    \Wass_2 (\mu_0,\mu_1) = 
    \inf_{
            (\mu,v) \in 
            \added[id=R1]{\mathcal A(\mu_0,\mu_1)}
    }
    \int_0^1 \int_\Rd |v_t|^2 \diff \mu_t \diff t,
\end{equation}
\added[id=R1, comment=the domain of the minimisation has been separated to another equation for better clarity]{%
    where 
    \begin{align*}
        \mathcal A (\mu_0, \mu_1) = \Big\{  (\mu, \nu) & : 
        \mu  \in C( [0,1], \mathcal P_2 (\Rd) ),  v \in L^2([0,1] \times \mathbb R^d,\mu)^d, \\ 
        & \qquad \mu(0) = \mu_0, \mu(1) = \mu_1
         \text{, and }
        \frac{\partial \mu_t}{\partial t} + \diver(\mu_t v_t) = 0 \Big\}
    \end{align*}
    is set space of admissible paths between $\mu_0$ and $\mu_1$.
}%

\subsubsection{Otto's calculus} 
The notion of calculus in these spaces is a little tricky, since  $\mathcal P_p(\Rd)$ are not vector spaces. However, they are subsets of the linear space of measures. 
In fact, it can be though of as a manifold: we can define tangent through curves. In particular, $\rho_0 \in \mathcal P_2(\Rd)$ and take $\varphi \in C_c^\infty(\Rd)$ a test function. Consider the curve of probabilities
$$
    \rho_t = ( \mathrm{id} + t \nabla \varphi )_\sharp \rho_0.
$$
It is not too difficult to see \cite[Lecture 16]{AmbrosioBrueSemola2021} that $\rho_t$ is a distributional solution of the continuity equation
$$
    \frac{\partial \rho_t}{\partial t} + \diver(  \rho_t v_t  ) = 0 
$$
where $v_t = \nabla \varphi \circ ( \mathrm{id} + t \nabla \varphi )^{-1}$. Then formally the tangent space $T_{\rho_0} \mathcal P_2(\Rd)$ is made of
$$
    s \defeq \frac{\diff}{\diff t}\Big|_{t=0} \rho_t = -  \diver(  \rho_t \nabla \varphi  ).
$$

This allows to formally construct the gradient of functions. Take the free energy of the diffusion term $\calU$.
In distributional sense the Wasserstein gradient is 
\begin{align*}
    \langle \nabla_{\Wass_2} U [\rho_0] , s \rangle &\defeq  \frac{\diff}{\diff t}\Big|_{t=0} \calU[\rho_t] = \int_\Rd U'(\rho_0 ) \frac{\partial }{\partial t}\Big|_{t=0} \rho_t \\
    &= - \int_\Rd U'(\rho_0 ) \diver(\rho_0 \nabla \varphi) 
    \\
    &=-\int_\Rd \varphi \diver( \rho_0 \nabla U'(\rho_0) ).
\end{align*} 
In general
\begin{equation} 
    \langle \nabla_{\Wass_2} \mathcal F[\rho_0] , s \rangle = 
    -\int_\Omega \varphi \diver\left( \rho \nabla \frac{\delta \mathcal F}{\delta \rho}[ \rho_0 ]\right)
\end{equation}
Hence, \eqref{eq:ADE} in $\Rd$ is formally written can be formally written as the $2$-Wasserstein gradient flow in the sense that
$$
    \frac{\partial \rho}{\partial t} = - \nabla_{\Wass_2} \mathcal F[\rho_t]
$$
of the free energy given by \eqref{eq:free energy}.

\begin{remark}
    \label{rem:Otto's calculus in bounded domains}
    If we work in a bounded domain, then we must study the space $\mathcal P_2 (\overline \Omega)$ and  understand how the no-flux condition \eqref{eq:BC Omega} appears naturally in this context.
    One approach is to consider $\varphi \in C_c^\infty(\overline \Omega)$ (meaning their support is bounded and does not reach the boundary). In order for $\mathrm{id} + t \nabla \varphi : \overline \Omega \to \overline \Omega$ when $t \in (-\varepsilon,\varepsilon)$ we must specify $\rho \nabla \varphi \cdot n = 0$.
    This is related to how the no-flux appears. We will continue this discussion below in \Cref{rem:JKO bounded domains}.
\end{remark}

\subsubsection{Rigorous gradient flow structure. The JKO scheme}
The idea of writing gradient flows in terms of minimising movements is due to De Giorgi (see \cite{DeGiorgi1993}). To draw a quick analogy think that we can to minimise a function $F:\mathbb R^d \to \mathbb R$. A continuous ODE going to local minimima is the $\Rd$ gradient flow
$$
    \frac{\diff x}{\diff t} = - \nabla F(x)
$$
The implicit Euler scheme leads to the implicit gradient-descent method
$$
    \frac{x_{n+1} - x_n}{\tau} = - \nabla F(x_{n+1}).
$$
We have avoided making explicit the dependence of $x_n$ with respect to $\tau$.  
We can rewrite this the minimisation of the so-called proximal function
\begin{equation}
    \label{eq:grad descent Rd}
    x_{n+1} = \argmin_{x \in \Rd} \left( \frac{|x-x_n|^2}{2} + \tau F(x) \right)
\end{equation}
If $F$ is smooth this minimisation problem has a unique solution for $\tau$ small. These are the so-called \emph{minimising movements}.
Then one can do a piece-wise constant interpolation $x^{(\tau)} (t) \defeq x_n$ when $t \in [\tau n, \tau (n+1))$, or even linear interpolation.

\bigskip 

The notion of minimising movement can be generalised to metric spaces and, in particular, we can construct in $2$-Wasserstein space
\begin{equation} 
    \label{eq:JKO}
    \mu_{n+1} = \argmin_{\mu \in \mathcal P_2(\Rd)} \left( \frac{\Wass_2(\mu,\mu_{n})^2}{2} + \tau \mathcal F(\mu) \right).
\end{equation}
Jordan, Kinderlehrer, and Otto proved in the seminal paper
\cite{JordanKinderlehrerOtto1998}
that this procedure works for \eqref{eq:Fokker linear}. The book \cite{AmbrosioGigliSavare2005} is devoted to proving how these minimising movements lead to \eqref{eq:ADE} in $\Rd$ for fairly general $\mathcal F$.
A good notion of solution for the limit of minimising movements is that of \emph{curves of maximal slope}. Often, it is even possible to get back a suitable solution of a PDE.

\begin{remark}
    \label{rem:JKO bounded domains}
    In bounded domains, we continue the comment in \Cref{rem:Otto's calculus in bounded domains}. In general, if one minimises over $\mathcal P_2(\overline \Omega)$, then we expect to arrive \eqref{eq:ADE} with the no-flux condition \eqref{eq:BC Omega star}.
    We recommend \cite[Chapter 8]{Santambrogio2015}.
\end{remark} 

However, the general setting is tricky as presented by the following example.

\begin{remark}
    \label{eq:transport in a bounded domain}
    Take $\Omega = (-1,1)$ (i.e., $d=1$), $U = W = 0$ and $V(x) = a x$. Then we recover the transport equation $\frac{\partial \rho}{\partial t} = a \frac{\partial \rho}{\partial x}$. The solutions of this equation are always of the form $\rho(t,x) = \rho_0(x+at)$. The boundary condition $a \rho = 0$ can be forced on one side, but never the other.
    From the perspective below one can always consider a Dirac delta on the right-hand side. This negates somehow $a\rho = 0$ on that side, but it is only possible solution. In fact, it is recovered from the vanishing viscosity formulation. 

    Notice that this examples does not satisfy \eqref{eq:V and K boundary condition}. If we replaced $V$ by a different smooth potential such that \eqref{eq:V and K boundary condition}, then the characteristics will not reach the boundary in finite time. Hence, we will expect a smooth solution.
\end{remark} 

\begin{remark}
    \added[id=O, comment=we added a comment on the doubly nonlinear equation and very recent reference]{%
    Otto's original paper \cite{Otto1996} already covered $p$-Wassersten distances with $p \ne 2$. In particular, it showed that the \emph{doubly nonlinear equation}
    \[
      \frac{\partial u}{\partial t} = \Delta_p u^m
    \]
    has a $p$-Wasserstein gradient-flow structure.
    Recently, a new article has appeared justifying the JKO scheme for these problems, see \cite{caillet2024DoublyNonlinearDiffusive}.
    }%
\end{remark}

\subsubsection{Convexity and McCann's condition}
\label{sec:convexity}
In the same way that \eqref{eq:grad descent Rd} works better if $F:\Rd \to \R$ is convex, there is a suitable notion of convexity that works for \eqref{eq:JKO}. The correct extension is convexity along geodesics also called \emph{displacement convexity}, i.e., if $\rho_t$ is a geodesic from $\rho_0$ to $\rho_1$ then
$
    t \mapsto \mathcal F[\rho_t]
$
is a convex $[0,1] \to \mathbb R$ function.
It is usually called \textit{displacement convexity} and was introduced by McCann in \cite{McCann1997}\footnote{It is interesting to mention that this work only mentions Wasserstein spaces in the title of one of the references, and rather speaks of a ``a novel but natural interpolation between Borel probability measures''}. 

\bigskip

Before we present this result, it is interesting to discuss the structure of geodesics in Wasserstein space (see, e.g., \cite{AmbrosioBrueSemola2021}). As the simplest case, assume that $\mu, \nu$ are absolutely continuous probability distributions and compactly supported. Then, Monge's problem \eqref{eq:Monge problem} for $p = 2$ admits an optimal transport map $T = \nabla \varphi$ such that $\nu = T_\sharp \mu$ (see, e.g., \cite[Lecture 5]{AmbrosioBrueSemola2021}). Furthermore, the unique geodesic between $\mu$ and $\nu$ is given by 
$$\rho_t = ((1-t) \mathrm{id} + t T)_\# \mu.$$

\bigskip

In \cite{McCann1997} the author proves that if $V, W$ are strictly convex, $W(x) = W(-x)$ then 
$$
    \int_\Rd V \rho , \qquad \iint_{\Rd \times \Rd} W(x-y) \rho(x) \rho(y) \diff x \diff y
$$  
are displacement convex. For the diffusion term there is a more interesting condition, known as \textit{McCann's condition}. It states that if we define
$$
    P'(s) \defeq s U'(s), \qquad P(0) \defeq 0
$$
then \eqref{eq:free energy U} is displacement convex if and only if
\begin{equation}
    \label{eq:McCann}
    P'(s) s \ge (1-1/d)P(s) \ge 0, \qquad \forall s \in (0,\infty).
\end{equation} 
This can also be written as $\lambda \mapsto \lambda^d U(\lambda^{-d})$ is convex.
When $U = \frac{s^m}{m-1}$ this means $m \ge \frac{d-1}{d}$.

\begin{remark}
    The case $m > \frac{d-1}{d}$ and $V(x) = \frac{|x|^2}{2}$ the functional is $0$-convex and bounded below. But recalling \Cref{eq:PME numerology 1} for $m < \frac{d}{2+d}$ the \eqref{eq:Barenblatt profile} is not in $\mathcal P_2$, we cannot hope for asymptotic convergence in 2-Wasserstein sense.
\end{remark}

\bigskip

Besides ``regular'' convexity, a powerful tool in the arsenal is the notion of $\lambda$-convexity. A function $f: (0,1)\to \mathbb R$ is $\lambda$-convex if $f(t) - \lambda \frac{|t|^2}{2}$ is convex. Analogously as above we introduce the notion of displacement (or geodesically) $\lambda$-convex. In this setting for the gradient flow starting from two points we have
$$
    \Wass_2(\mu_t, \widehat \mu_t) \le e^{-\lambda t} \Wass_2(\mu_0, \widehat \mu_0).
$$
This implies uniqueness of the gradient flow and, if $\mathcal F$ is bounded below, of its minimisers. 

\bigskip 

A similar analysis of a free energy in $\mathbb S^1$ is done in \cite{CarrilloSlepcev2009}.
The sense in which the equation is satisfied is delicate. For $U = V = 0$ and $W$ ``pointy'', we send the reader to \cite{CarrilloDiFrancescoFigalliLaurentSlepcev2011}.

	\subsection{Global existence vs finite-time blow-up}

As explained above, local-in-time existence is usually achieved through ``standard'' PDE methods or the use of gradient flow arguments, specially if starting from a nice initial datum $\rho_0$. 
Whether these solutions behave ``nicely'' for large $t$ is a more difficult problem. 

\paragraph{Transport for given potential} Going back to the simplest example \eqref{eq:ADE U,W = 0} we can look at the examples $V = \frac{|x|^\alpha}{\alpha}$ with $\alpha > 0$. 
Then $v = -\nabla V = - |x|^{\alpha-2} x$ and $\rho_t = (X_t)_\sharp \rho_0$ with corresponding characteristic field is obtained by solving
$$
    \frac{\partial X_t}{\partial t} = -|X_t|^{\alpha-2} X_t,
$$
with $X_t(0) = y$. We conclude that 
$$
    X_t (y) = 
    \begin{dcases} 
        y e^{-t} & \text{if }\alpha=2, \\
        ( |y|^{2-\alpha} - (2-\alpha) t  )^{\frac 1 {2-\alpha}} \frac{y}{|y|} & \text{if } \alpha \ne 2.
    \end{dcases}
$$
For $\alpha < 2$ the characteristics arrive at $0$ at finite time and hence we get a Dirac delta at $0$ at a certain finite time for any $\rho_0 \not\equiv 0$. 
This is not problematic in the distributional sense in this case. However, this behaviour happens also for other problems.

\paragraph{Keller-Segel problem.} Some authors noticed by direct techniques that the Keller-Segel model has finite-time blow-up for initial data with large enough mass.
\cite{JaegerLuckhaus1992} showed for \eqref{eq:Keller-Segel un-normalised} in a bounded domain with no-flux the existence of a critical mass $M^*$ such that, if $\| \rho_0 \|_{L^1} > M^* $, 
then $\rho(T^*)$ contains a Dirac delta. Later this result was obtained also in $\Rd$ in \parencite{HerreroVelazquez1996} using radial arguments, where the authors characterise the critical mass is $M^* = 8 \pi$. This was later proved in more generality in \cite{DolbeaultPerthame2004a}.
Conversely, for $\| \rho_0 \|_{L^1} < M^*$ global existence for \eqref{eq:Keller-Segel un-normalised} is known. In $d = 2$ it was shown in \cite{BlanchetDolbeaultPerthame2006}. 

\bigskip

A nice proof of the formation of a Dirac comes through the analysis of the second-order moment.
For \eqref{eq:Keller-Segel un-normalised} in $d = 2$ it was first noticed by 
\cite{DolbeaultPerthame2004a}
that, working with solutions of initial mass $\int_\Rd \rho_0 = M$,
$$
\frac{\diff}{\diff t} \int_\Rd |x|^2 \rho(t,x) \diff x 
= 
4 M \Big( 1- \frac{M}{8\pi }\Big).
$$
Hence, if $M > {8 \pi}$ necessarily the second-order moment becomes negative in finite time. This is incompatible with our non-negative solutions. There is complete concentration to a Dirac delta. This was later extended by \cite{BlanchetCarrilloLaurencot2009} for \eqref{eq:Keller-Segel un-normalised} when $d > 2$ where the authors prove that
\begin{equation*}
    \frac \diff {\diff t} \int_\Rd |x|^2 \rho(t,x) \diff x 
    = 
    2 (d-2) \mathcal F[\rho(t,\cdot)]  
    \le 
    2 (d-2) \mathcal F[\rho_0].
\end{equation*} 
When $M > M^*$, there exist $\| \rho_0 \|_{L^1} = M$ such that $\mathcal F[\rho_0] < 0$. 
We will perform the analysis of the free energy for this problem in \Cref{rem:scaling free energy KS}.
We will discuss below in \Cref{eq:KS bubbles} the critical case $M = 8\pi$ and some interesting \textit{ad-hoc} constructions which are available in the literature.

\section{Asymptotics}
	The aim of this section is to highlight some techniques coming from the PDE community that allow us to understand the behaviour as $t \to \infty$ (and sometimes $|x| \to \infty$) of ``typical'' solutions of the problem (namely those with ``good'' initial values).
	\subsection{Self-similarity}

Some of the equations we are analysing have linear operators \emph{homogeneous non-linearities}, and it is therefore interesting to exploit this structure. 
One of the main approaches comes from considering the mass-preserving change of variable
\renewcommand{\unk}{u}
\begin{equation} 
    \label{eq:mass preserving change of variable}
        \unk(\tau,y) = \sigma(\tau)^{d} \rho\Big(\tilde t(\tau), \sigma(\tau)y\Big).
\end{equation}

\subsubsection{Self-similar scaling for \texorpdfstring{\eqref{eq:PME}}{(PME)}}
\label{sec:self-similar PME}

	Let $\rho$ be the solution of \eqref{eq:PME}.
    Applying the chain rule to \eqref{eq:mass preserving change of variable}                      
    \begin{align*}
        \frac{\partial \unk}{\partial \tau} 
        &=  \frac{\frac{\diff \sigma}{\diff \tau}}{\sigma} \unk \diver_y y +  \sigma^{d}\frac{\diff \sigma}{\diff \tau}   \nabla_x \rho \cdot \nabla_y y + \sigma^{d}\frac{\diff \tilde t}{\diff \tau} \Delta_x \rho^m
    \end{align*}
    since $\nabla_x x = (1, \cdots, 1) = \nabla_y y$ y $\diver_x x = d = \diver_y y$.
    Due to the chain rule
    \begin{align*} 
        \frac{\partial \unk}{\partial y_i} &= \sigma^{d+1} \frac{\partial \rho}{\partial x_i} , \qquad \qquad 
        \frac{\partial^2 \unk^m}{\partial y_i^2}  = \sigma^{2+md} \frac{\partial^2 \rho^m}{\partial x_i^2}
    \end{align*}
    Eventually, we simplify the equation for $u$ to
    \begin{align*}
        \frac{\partial \unk}{\partial \tau} 
        &= \frac{\frac{\diff \sigma}{\diff \tau}}{\sigma} \diver_y(\unk y) + \sigma^{-d(m-1) + 2} \frac{\diff \tilde t}{\diff \tau} \Delta_x \rho^m 
    \end{align*} 
    The only sensible choice is $\frac{\diff \sigma}{\diff \tau} = \sigma$ and $\frac{\diff \tilde t}{\diff \tau} = \sigma^{\frac 1 \kappa}$ where 
    $\kappa = (d(m-1)+2)^{-1},$
    so we arrive to the simplified equation
    \begin{align*}
        \frac{\partial \unk}{\partial \tau} 
        &=  \diver_y(\unk y) + \Delta_y \unk^m.
    \end{align*} 

    \paragraph{Forward self-similar solutions}
    
    Notice that adding the condition $\sigma(0) = 1$ we get $\sigma(\tau) = e^\tau$.
    We also have that $\frac{\diff \tilde t}{\diff \tau} = e^{\tau / \kappa}$. We have $t = \tilde t(\tau)$ so we set $\tilde t(0) = 0$. Eventually
    \begin{align*} 
        t &= \frac{ e^{\frac \tau \kappa} - 1  }{\kappa} \\
        \kappa &= (d(m-1)+2)^{-1} \\
        x &= e^\tau y.
    \end{align*}
    Notice that $\kappa > 0$ if and only if $m > m_c = \frac{d-2}{d}$.   
    The Barenblatt solution \eqref{eq:Barenblatt} corresponds to setting $u$ as stationary. 
    In fact, for $m > m_c$ we can write the family of self-similar solutions
    \begin{align*}
        \rho_t(x) &\defeq \frac{1}{\sigma_B(t+T)^{d}} F\left( \frac{x}{\sigma_B(t+T)}  \right), \\
        F_B(y) &\defeq \left(D - \frac{m-1}{2m}|y|^2\right)^{\frac 1 {m-1}}_+\\
        \sigma_B(t) &\defeq (  d|m-m_c|t  )^{\frac 1 {d(m-m_c)}}.
    \end{align*}
    Here we have chosen the presentation in \cite{carrillo2000Asymptotic1decaySolutions,blanchet2009AsymptoticsFastDiffusion}.

\paragraph{Backwards self-similar solutions}
    When $\kappa < 0$ (i.e., $m > m_c = \frac{d-2}{d}$), we can build another type of solution.
    In \cite[Section 5.2]{Vazquez2006}, Vázquez introduces the pseudo-Barenblatt solutions for the sub-critical regime $0 < m < m_c$
    which can be written
    \begin{align*}
        \rho_t(x) &\defeq \frac{1}{\sigma_B(T-t)^{d}}  F_B\left( \frac{x}{\sigma_B(T-t)}  \right). \\
    \end{align*}
    Here, $T$ denotes the extinction time.
    They are suitable to deal with finite-time extinction \cite{blanchet2009AsymptoticsFastDiffusion}.

\subsubsection{Self-similar analysis in more general contexts}

In the previous computation we can trade $\Delta_x \rho^m$ for any non-linear operator $\mathcal L$ with the scaling property
$$
    \mathcal L [u] = \sigma^q \mathcal L[\rho].
$$
Self-similarity analysis for 
\eqref{eq:Caffarelli-Vazquez} was performed in \cite[Theorem 3.1]{CaffarelliVazquez2011b}. The result passes by an obstacle problem. 
For \eqref{eq:Newtonian vortex} the self-similar analysis is done in \cite[Section 6.2]{SerfatyVazquez2014}.
For the fractional heat equation see \cite{gentil2008LevyFokkerPlanck}.

\bigskip 

Self-similar solutions are very particular of problem where all terms have the same scaling, and it is not reasonable to expect solutions of this nature to exist in general.
Consider the problem
$$
    \frac{\partial \rho}{\partial t} = \mathcal L \rho. 
$$
When $\mathcal L$ is the sum of two terms with different scaling, it is not reasonable to expect self-similar structure.
For \eqref{eq:ADE} with $U = U_m, V = 0, W=W_k$ this is only possible in the so-called \emph{fair-competition regime}
$$
    d(m-1) + k = 0.
$$
The self-similar analysis in this setting was performed in \cite{calvezEquilibriaHomogeneousFunctionals2017}.

\subsubsection{The Keller-Segel problem: bubbles}
\label{eq:KS bubbles}
For $d = 2$ a self-similar solution is not available, although \cite{HerreroVelazquez1996} give some hints on the shape by matching asymptotics. Later, there were more advanced studies on stability \cite{velazquez2002StabilityMechanismsChemotactic,velazquez2004PointDynamicsSingular,velazquez2004PointDynamicsSingulara}.
In this direction, see also \cite{raphael2014StabilityCriticalChemotactic,collot2022RefinedDescriptionStability}.

In 
\cite{herrero1998SelfsimilarBlowupReactiondiffusion} the authors prove there are no self-similar solution of \eqref{eq:Keller chi} when $d = 2$. 
However, for $d = 3$ they show that there exists a sequence of radial self-similar solutions $\rho^{(n)}$ 
$$
    \rho^{(n)}_t (x) = \frac{1}{T-t} F_n \left(\frac{x}{\sqrt{T-t}}\right),\qquad \qquad  \text{where } F_n(y) \sim \left( \frac{8\pi}{\chi} + a_n \right)\frac{1}{4 \pi |y|^2} \text{ for } |y| \sim 0,
$$
and $a_n \to 0$.
This construction was later generalised for $d \ge 3$ by several authors \cite{senba2005BlowupBehaviorRadiala,giga2011AsymptoticBehaviorType,souplet2019BlowupProfilesParabolic}.
\added[id=O]{Blow-up on the range $U = U_m$ for $m \ge 1$, $V = 0$ and $W = W_k$ with $k > 2-d$ see \cite{YaoBertozzi2013BlowupDynamicsAggregation}.}

\bigskip

It was later conjectured in \cite{delpino2017BubblingBlowUpCritical} and proved in \cite{davilaExistenceStabilityInfinite} that for $d = 2$ one can construct
``approximate bubbles''. In particular,
 an initial datum $\rho_0^\star$ with initial mass $8 \pi$ so that the solutions of Keller-Segel problem \eqref{eq:Keller-Segel un-normalised} such that, for $\rho_0$ close to $\rho_0^\star$
$$
    \rho(t,x) \approx \frac{1}{\sigma(t)^2} F\left( \frac{x-\chi(t)}{\sigma(t)} \right), \qquad F(y) = \frac{8}{(1+|y|^2)^2}
$$  
as $t \to \infty$
where $\sigma(t) \approx c/ \sqrt{\log t}$, $\chi(t) \to q$.

	\subsection{Relative entropy} 
\label{sec:relative entropy}

If we are able to construct \emph{a} global solution $B_t$ of which we know some properties, we would like to see whether this is the ``generic behaviour''.
Assume we are in a case where $\| \rho_t \|_{L^1}$ is preserved. We would say that $B_t$ is attractive if
$$
    \| \rho_t - B_t \|_{L^1} \to 0, \qquad \text{as } t \to \infty.
$$
For $L^p$ norms in general (which typically change over time), we would to see if the \emph{relative error} tends to zero
\begin{equation}
    \label{eq:intermediate asymptotics}
    \frac{\| \rho_t - B_t \|_{L^p}}{\| \rho_t \|_{L^p}} \to 0, \qquad \text{as } t \to \infty.
\end{equation}
These are usually called \emph{intermediate asymptotics}, as opposed to the \emph{long-time asymptotic} limit.

\bigskip

Naïve computations do not yield good results.
Sometimes it is possible to do this computation directly via comparison arguments. But, in general, this is not possible. 
A useful tool is are the so-called \emph{relative entropies}.
We start by a simple example, \eqref{eq:HE}. 

\subsubsection{\texorpdfstring{$L^2$}{L2} relative entropy for \texorpdfstring{(\ref{eq:HE})}{(HE)}}
An interesting alternative way of writing \eqref{eq:Fokker linear}
\added[id=R1,comment=we thank the referee for the thoughtful comments on the relation between FP and Ornstein-Uhlenbeck]
{%
    as
    \begin{equation*}
        \frac{\partial u}{\partial t} = \diver\left(G \nabla \frac{u}G\right),
    \end{equation*}
    where $G$ is the Gaussian profile \eqref{eq:Gaussian}.
    Let us take $w = \tfrac u G$ we rewrite this problem as
    \begin{equation*}
        \frac{\partial w}{\partial t} = G^{-1} \diver \left( G \nabla w \right) = \Delta w - x \cdot \nabla w.
    \end{equation*}
    This is the famous Ornstein-Uhlenbeck problem, which is also the dual \eqref{eq:Fokker linear}.
}%
We can write the free-energy dissipation formula
$$
    \frac{\diff}{\diff t} \int_\Rd |w-1|^2 
    \added[id=R1]{G} 
    \diff x = - 2 \int_\Rd |\nabla w|^2 G\diff x.
$$
We can now take advantage of the Gaussian Poincaré inequality%
\added[id=O]{%
    , see \cite{Chernoff1981NoteInequalityInvolving},
}%
$$
    \int_\Rd w^2 G\diff x - \left( \int_\Rd w G \diff x \right)^2 \le \int_\Rd |\nabla w|^2 \added[id=R1]{G} .
$$
Hence, 
\added[id=R1]{%
    assuming that
    $
        \int_\Rd w \added[id=R1]{G = \int_\Rd u} = 1,
    $
}%
we recover that
$$
    \int_\Rd |w_t - 1|^2 G\diff x \le e^{-2t} \int_\Rd |w_0 - 1|^2 G \diff x.
$$
This approach can be generalised in many directions. We point the reader to \cite{ArnoldMarkowichToscaniUnterreiter2001} 
\added[id=R1]{%
and the references therein.
}%

\subsubsection{\texorpdfstring{$L^1$}{L1} Relative entropy argument for (\ref{eq:HE})} 
\renewcommand{\unk}{v}
For convenience, we will denote the solution of \eqref{eq:Fokker linear} by $\unk$. Notice that, with our choice of re-scaling $v_0 = \rho_0$.
The relative entropy is defined 
\begin{equation*}
    \mathcal E (\unk) = \int_\Rd \unk \log \frac{\unk}{G} = \int_\Rd \unk \left( \log \unk + \frac{x^2}{2} \unk \right)  + C.
\end{equation*}
It is easy to check that
$$
    \frac{\diff}{\diff t} \mathcal E(\unk) = - \mathcal I (\unk),
$$
where $\mathcal I$ is the so-called \textit{Fisher information}
\begin{equation*}
    \mathcal I (\unk) = \int_\Rd \unk \left| \frac {\nabla \unk }{\unk} + x \right|^2 = \int_\Rd \unk \left| \nabla \log \frac \unk G \right|^2 .
\end{equation*}
The Gaussian logarithmic Sobolev inequality (see \cite{Gross1975,Toscani1999,DolbeaultToscani2015})
\begin{equation}
    \label{eq:Gaussian log-Sobolev}
    \frac 1 2 \int_\Rd (|f|^2 \log |f|^2) G \diff x \le \int_\Rd |\nabla f|^2 G, \qquad \text{if } \int_\Rd |f|^2 G = 1.
\end{equation}
This inequality ensures the relationship
$\mathcal E \le \tfrac 1 2 \mathcal I$,
and hence we recover an ordinary differential inequation that yields $\mathcal E(\unk_t) \le \mathcal E(\rho_0) e^{-2t}$.
Lastly, we will take advantage of  the Czsizar-Kullback inequality taking $G = K_1$
\begin{equation*}
    \| f - G \|_{L^1}^2 \le 2 \mathcal E(f).
\end{equation*}
Eventually, we deduce that if $\rho$ is a solution \eqref{eq:HE}
\begin{equation*}
    \| \rho_t - K_t \|_{L^1} \le \sqrt {2 \mathcal E(\rho_0)} t^{-\frac 1 2}.
\end{equation*}

\begin{remark}
    It is also worth point out that \eqref{eq:Gaussian log-Sobolev} is equivalent to the Euclidean log-Sobolev inequality, which in scale invariant form (see \cite{weissler1978LogarithmicSobolevInequalities,DolbeaultToscani2015}) given by
    \begin{equation}
        \label{eq:log-Sobolev}
        \frac{d}{2} \log\left( \frac{2}{\pi d e} \int_\Rd |\nabla f|^2   \right)  \ge \int_\Rd |f|^2 \log |f|^2 , \qquad \text{if } \int_\Rd |f|^2 = 1.
    \end{equation}
    Notice that is usually applied to $f = \sqrt \rho$, and hence the condition is simply that $\rho \in L^1 (\Rd) \cap \Prob (\Rd)$.
\end{remark}

\subsubsection{\texorpdfstring{$L^1$}{L1} relative entropy for (\ref{eq:PME})}
\label{sec:relative entropy PME}

In 
\added[id=R1]{%
    \cite{carrillo2000Asymptotic1decaySolutions,delpino2002BestConstantsGagliardo} 
}%
the authors extend the relative-entropy study to $m > 1$.
Now we denote by $\unk$ the solution of \eqref{eq:Fokker PME}.
Suppose there exists $\hat \unk$ of suitable mass and define the relative entropy and the Fisher information
$$
    H(\unk) = \int_\Rd \left(   \frac{2}{m-1} \unk^m  + |x|^2 \unk    \right) 
\qquad \qquad
    I(\unk) = \int_\Rd \unk\left|  x + \frac{m}{m-1} \nabla \unk^{m-1}  \right|^2.
$$
The suitable relative entropy for $m > 1$ is $\mathcal H(u||\hat u) = \mathcal H (u) - \mathcal H(\hat u)$.
This was later extended to more general families than \eqref{eq:Fokker PME} in \cite{CarrilloJuengelMarkowichToscaniUnterreiter2001}.
This relative-entropy arguments rely heavy on Gagliardo-Nirenberg-Sobolev inequalities. In fact, there is a deep connection between the smoothing properties of \eqref{eq:PME} and these types of inequalities (see, e.g., \cite{delpino2002BestConstantsGagliardo}).
In
\cite{blanchet2009AsymptoticsFastDiffusion,bonforte2010SharpRatesDecay} the authors develop the correct relative entropies and Fisher information functions to cover the cases $m < 1$ with the correct re-scaling of the Barenblatt/pseudo-Barenblatt explained in \Cref{sec:self-similar PME}.

\begin{remark}
    This kind of relative entropy arguments for diffusive problems (which do not have a natural stationary state) rely on the self-similar solution.
    It is possible to prove intermediate asymptotics without homogeneity or self-similarity techniques. 
    For diffusive problems of the form $\frac{\partial \rho}{\partial t} = \Delta \Phi (\rho)$, a very nice analysis using Wasserstein spaces can be found in
    \cite{carrillo2006IntermediateAsymptoticsHomogeneity}.
\end{remark}

\subsubsection{Relative entropy argument for (\ref{eq:ADE})}
\label{sec:relative entropy for ADE}

The entropy arguments for \eqref{eq:PME} are stable enough that allow some families of perturbations into the range \eqref{eq:ADE}. 
For $W = 0$ and confining potentials see \cite{CarrilloJuengelMarkowichToscaniUnterreiter2001}.
The reader may find information on the strictly displacement-convex free-energy functionals \cite{CarrilloMcCannVillani2003,CarrilloMcCannVillani2006}.
For \eqref{eq:Keller PME} with $m \in [1,2-\frac 2 d]$ (and, in fact, a larger family of $W$ radially decreasing), intermediate asymptotics in the sense \eqref{eq:intermediate asymptotics}  are obtained in \cite{bedrossian2011IntermediateAsymptoticsCriticala}.
In \cite{Carrillo+GC+Yao+Zheng2023} the authors take advantage of relative entropy arguments to prove that in the case $U(\rho) = \rho \log \rho, V = 0$ and $W$ smooth and bounded \eqref{eq:intermediate asymptotics} with $p = 1$ and $B_t = K_t$ (i.e., the heat kernel).
See \cite{CanizoCarrilloSchonbek2012} for previous results in this direction.

\subsection{Study of the mass variable of radial solutions}

One of the significant limitations to prove the above results rigorously is the lack of regularity of $\rho$. Some regularity can be regained by passing to the so-called \emph{mass variable}.

\bigskip

Let $d = 1$. Assume $x \in \mathbb R$ and 
\begin{equation}
    \label{eq:mass of rho}
M(t,x) = \int_{-\infty}^x \rho_t(y) \diff y.
\end{equation}
Integrating \eqref{eq:ADE} we show that $M$ satisfies a PDE
$$
\begin{aligned}
    \frac{\partial M}{\partial t} &=   \Phi' \left (\frac{\partial M}{\partial x} \right )  \frac{\partial^2 M}{\partial x^2 } + \frac{\partial M}{\partial x} \frac{\partial V}{\partial x} +  \frac{\partial M}{\partial x} \frac{\partial }{\partial x} \left(W*\frac{\partial M}{\partial x}\right) 
\end{aligned} 
$$
If $\Phi(s) = |s|^{m-1}s$ (PME for $\rho$), this is the $p$-Laplacian for $M$ with $p = m+1$. 
If $x \in \mathbb R^d$ with $d > 1$ it is better to define the mass in volumetric coordinates
$$
M(t,v) = \int_{A_v}  \rho(t, x) \diff x, \qquad \text{where } A_v = B(0,r) \text{ such that } |A_v| = v. 
$$
Then, if $\rho$ is radially, we recover for a weight $\kappa$
\begin{align*}
    \frac{\partial M}{\partial t} &=  \kappa(v)^2  \Phi' \left (\frac{\partial M}{\partial v} \right )  \frac{\partial^2 M}{\partial v^2 } + \kappa(v)^2\frac{\partial M}{\partial v} \frac{\partial V}{\partial v} + \kappa(v)^2  \frac{\partial M}{\partial v} \frac{\partial }{\partial v} \left(W*_{\Rd}\frac{\partial M}{\partial v}\right) 
\end{align*} 
This is a Hamilton-Jacobi type equation, and is well suited for the theory of viscosity solutions.
In \cite{Carrillo+GC+Vazquez2022JMPA} the authors discuss the problem when $W = 0$. They show that one can take advantage of stability properties of viscosity solutions to characterise the steady state. 
Since $\kappa(0) = 0$ and $\kappa(v) > 0$ for $v>0$, this equation can be used to show that for smooth $V$ and $W$ and radially symmetric $\rho$, singularities can only form at $x = 0$. They also show that a Dirac may form, but only in the limit $t \to \infty$. 

\bigskip 

Many authors have taken advantage of this kind of idea in their arguments. For example, we point to
\cite{herrero1998SelfsimilarBlowupReactiondiffusion,senba2005BlowupBehaviorRadiala,giga2011AsymptoticBehaviorType,KimYao2012}.

\bigskip 

There is also an interesting connection between the mass variable and the Wasserstein distance for $\rho$. The simplest case states that if $\rho_1, \rho_2 \in \mathcal P(\R)$, and $M_i$ are their primitives, then 
$$
    \Wass_1(\rho_1, \rho_2) = \| M_1 - M_2 \|_{L^1 (\R)}.
$$
A similar formula holds if $\rho_i \in \mathcal P(\Rd)$, and are radially symmetric.

\paragraph{The inverse of the mass function}
Furthermore, if we consider the generalised inverse
$$
    M^{-1}(s) = \inf\{ v \in [-\infty,\infty) : M(v) \ge s\}
$$
and $\rho_1, \rho_2 \in \mathcal P(\R)$, and $M_i$ are their primitives, then 
$$
    \Wass_p(\rho_1, \rho_2) = \| M_1^{-1} - M_2^{-1} \|_{L^p (0,1)}.
$$
Interestingly, if $\rho(t,x)$ solves a conservation problem, we can also deduce and equation for $u(t,s) = M^{-1} (t,s)$. This equation is usually degenerate, but the coefficients do not depend on $s$ (see, e.g., \cite{CarrilloHopfRodrigo2020}).
	\subsection{On the attractiveness of attractors}

Relative entropy arguments allow us to show that the steady state (or energy minimiser) in $L^1$ is an attractor for a large range of equations: from \eqref{eq:Fokker PME} to all the examples discussed in \Cref{sec:relative entropy for ADE}.

\bigskip 

A more difficult question is to understand the case where the energy minimiser may contain a singular part, e.g.,
\begin{equation}
    \label{eq:steady state with concentrated part}
    \widehat \mu = \widehat \rho + (1-\|\widehat \rho\|_{L^1}) \delta_0 ,
\end{equation}
and $\|\widehat \rho\|_{L^1} < 1$. 
We will see an example in the aggregation dominated range (see \Cref{cor:Carrillo Delgadino} below). 
It is difficult to say whether a measure of this kind can be called a \emph{steady state}, since it cannot be plugged into even the distributional formulation.
If the energy is $\lambda$-displacemente convex, there is no doubt that the minimiser is a global attractor. 

\bigskip 

In \cite{Carrillo+GC+Vazquez2022JMPA} the authors prove that \eqref{eq:ADE Omega} $m \in (0,1)$, $V$ smooth, and $K = 0$, the solution converges to the free energy minimiser, even when this contains a singular part.
In \cite{Carrillo+Fernandez-Jimenez+GC2023} the authors use the mass to prove that for \eqref{eq:ADE Omega} $m \in (0,1)$, in large class of $K$.

\bigskip 

In the recent preprint \cite{bian2023AggregationdiffusionEquationEnergy}, covering \eqref{eq:ADE} when $k = 2 - 2s \in (2-d,0)$ and $m = \frac{2d}{d+2s}$ (part of the \emph{fair-competition regime} described in \Cref{sec:minimisation of power cases}) shows that there is an explicit steady state (given in \eqref{eq:steady state conformal fractional} below), and solutions with $\|\rho_0\|_{L^m} > \|\widehat \rho\|_{L^m}$ have $\|\rho_t\|_{L^m}$ blow-up in finite time (and hence even we cannot study them in the distributional sense).

\section{Minimisation of the free energy}

As mentioned in \Cref{sec:convexity}, displacement $\lambda$-convexity with $\lambda > 0$ guarantees exponential contraction with rate
$$
    \Wass_2(\mu_t, \widehat \mu_t) \le e^{-\lambda t} \Wass_2(\mu_0, \widehat \mu_0)
$$
In particular, if $\widehat \mu$ is a global minimiser of $\mathcal F$ (which is unique by convexity), then the JKO schemes is stationary and so is the gradient flow. This means that 
\[
    \Wass_2(\mu_t, \widehat \mu) \le e^{-\lambda t} \Wass_2(\mu_0, \widehat \mu)
\]
Therefore, it is a global attractor with exponential rate.
First, we will try to characterise the local minimisers using variational arguments.
\added[id=O]{%
    This will lead us to Euler-Lagrange conditions, and generally corresponds to $2$-Wasserstein local minimisers.
    On each connect component of the support of a steady state, the Euler-Lagrange condition comes also from the free-energy dissipation \eqref{eq:free-energy dissipation}, and are simply $\frac{\delta F}{\delta \rho} = C$. 
}%

\bigskip

\added[id=O, comment=This paragraph motivates the use of $\infty$-Laplacian below]{%
    However, \eqref{eq:ADE} can have steady states that do not satisfy the Euler-Lagrange conditions, specially when the support has several components.
    A famous example that is known to the community is the case $U = U_m$ with $m > 1$, $ W = 0$, and a potential $V$ with two wells, e.g., $V(x) = |x-x_0|^2 + |x-x_1|^2$. 
    Then we can build formal steady states
    \begin{equation}
        \label{eq:PME+2well steady states}
        \hat \rho (x) = \Big((U_m')^{-1}(h_1 - V(x))\Big)_+ + \Big((U_m')^{-1}(h_2 - V(x))\Big)_+ \text{ even with } h_1 \ne h_2,
    \end{equation}
    provided the support of these two terms are disjoint. These are not 2-Wasserstein local minimisers, but they are, actually, $\infty$-Wasserstein local minimisers.
    Furthermore, these steady states attract some initial data.
    The study of $\infty$-Wasserstein local minimisation became quite popular, and we cite several references below.
}%

	\subsection{Local minimisers: Euler-Lagrange conditions}
\label{sec:Euler-Lagrange}

The aim of this section is show that if $\hat \rho \in L^1 (\Rd) \cap \mathcal P (\Rd)$ is a local minimiser of $\calF$ over 
\added[id=R1]{%
    $L^1 (\Rd) \cap \mathcal P (\Rd)$ 
}%
then we expect it to satisfy the following: there exists $h \in \mathbb R$ such that
\begin{align}
    \label{eq:Euler-Lagrange 1}
    \tag{EL$_1$}
    &\frac{\delta \calF}{\delta \rho}[\hat \rho] (x)   = h  \qquad \text{a.e. in } \supp \widehat \rho \\
    \label{eq:Euler-Lagrange 2}
    \tag{EL$_2$}
    &\frac{\delta \calF}{\delta \rho}[\hat \rho] (x)   \ge h  \qquad \text{a.e. in } \Rd
\end{align}
We recall that our main interest is the free energy $\mathcal F$ given by \eqref{eq:free energy star} for which the first variation $\frac{\delta \calF}{\delta \rho}$ is given by \eqref{eq:first variation}.

\subsubsection{First variation over the space of absolutely continuous probability measures}
\label{sec:first variation probability}
We follow the argument of 
\added[id=R1]{%
    \cite{BlanchetCarrilloLaurencot2009,balague2013DimensionalityLocalMinimizers,canizo2015ExistenceCompactlySupported,CarrilloCastorinaVolzone2015,calvezEquilibriaHomogeneousFunctionals2017,CarrilloHoffmannMaininiVolzone2018,CarrilloHittmeirVolzoneYao2019,Carrillo+GC+Vazquez2022JMPA} 
}%
where the reader may find rigorous proofs in different ranges. Our aim is to take variations of the forms
$$
    \widehat \rho_\ee = \widehat \rho + \ee \varphi
$$
where, if we pick test functions $\varphi$ suitably, we still have $\widehat \rho_\ee \in L^1 \cap \Prob$ so
$$
    \calF(\widehat \rho) \le \calF (\widehat \rho_\ee).
$$

\paragraph{Variations on the support of $\hat \rho$} 
Consider a test function $\psi \in C_c^\infty (\Rd)$ and introduce 
$$
    \varphi(x) = \left(  \psi(x) - \int_\Rd \psi(y) \hat \rho(y) \diff y    \right) \hat \rho(x) 
$$
Then for $\ee < \frac 1 2 \|\psi \|_{L^\infty}^{-1}$ we have $\hat \rho_\ee \in L^1 \cap \mathcal P_2$ and, using that $\widehat \rho$ is a minimiser and the definition of first variation, we get
\begin{align*}
    \int_\Rd \frac{\delta \calF}{\delta \rho}[\hat \rho] \varphi =  \lim_{\ee \to 0} \frac{1}{\ee} \left( \calF(\rho_\varepsilon) - \calF(\rho) \right) \ge 0.
\end{align*}
Expanding this integral and using $\| \hat \rho \|_{L^1} = 1$ we get
\begin{align*}
    0 &\le \int_\Rd \frac{\delta \calF}{\delta \rho}[\hat \rho] (x) \left(  \psi(x) - \int_\Rd \psi(y) \hat \rho(y) \diff y    \right) \hat \rho(x)   \diff x\\
    &=
     \int_\Rd \frac{\delta \calF}{\delta \rho}[\hat \rho] (x)  \hat \rho(x) \psi (x) \diff x -  \left( \int_{\Rd} \frac{\delta \calF}{\delta \rho}[\hat \rho] (x) \hat \rho(x) \diff x \right) \left( \int_{\Rd} \psi (y) \hat \rho(y) \diff y \right) 
\end{align*}
Lastly, trading $x$ and $y$ on the last two integrals we recover that for all $\psi \in C_c^\infty(\Rd)$ we have
$$
\int_\Rd \left( \frac{\delta \calF}{\delta \rho}[\hat \rho] (x)  -  h(\hat \rho)  \right) \psi(x) \hat \rho(x) \diff x \ge 0   , \qquad \text{where } h(\hat \rho) =  \int_{\Rd} \frac{\delta \calF}{\delta \rho}[\hat \rho] (y) \hat \rho(y) \diff y .
$$
Since this also holds for $-\psi$ the equality holds above, and hence 
\begin{equation*} 
    \left( \frac{\delta \calF}{\delta \rho}[\hat \rho] (x)  -  h(\hat \rho)  \right) \widehat \rho(x) = 0, \qquad \text{for a.e. } x \in \Rd 
\end{equation*} 
Hence we recover \eqref{eq:Euler-Lagrange 1}. This gives us complete information on the boundary of the support. But it could jump uncontrollably from $0$ to the other profile.

\paragraph{Variations outside the support of $\hat \rho$} 
Now we take the variations 
$$
    \varphi(x) = \psi(x) - \hat \rho (x) \int_{\Rd} \psi(y) \diff y.
$$
If $\psi \ge 0$ and $\ee < \|\psi\|_{L^1}^{-1}$ then $\hat \rho_\ee \ge 0$. Notice that we cannot hope for $\psi < 0$ outside the support of $\hat \rho$. 
With a similar argument as before we get
$$
    \int_\Rd \psi \left( \frac{\delta \calF}{\delta \rho}[\hat \rho] (x)  -  h(\hat \rho) \right) \ge 0 .
$$
Since this holds true for all $\psi$ is the conditions above, we have proven \eqref{eq:Euler-Lagrange 2}. 

\begin{remark} 
    For \eqref{eq:ADE} we have that
    \begin{align*} 
        h(\hat \rho) &
        = \int_\Rd (U'(\hat \rho)\hat \rho + V \hat \rho + (W*\hat \rho) \hat \rho ) 
    \end{align*}
    With PME diffusion
    \begin{equation*}
        h(\hat \rho) = \mathcal F(\hat \rho) + (m-1) \int \rho^m + \frac 1 2 \int_\Rd (W*\hat \rho) \hat \rho
    \end{equation*}
\end{remark} 

\subsubsection{Solving the Euler-Lagrange equation when \texorpdfstring{$U \ne 0$}{U ≠ 0}}

When $U \not \equiv 0$ then $U'$ is non-decreasing, and we get
\added[id=O]{%
    from \eqref{eq:Euler-Lagrange 2}
}%
\begin{equation}
    \label{eq:EL outside support U ne 0}
    \hat \rho  \ge (U')^{-1} ( h - V - W* \hat \rho ) 
\end{equation} 
we also know that $\hat \rho \ge 0$, and 
\added[id=R1]{%
    \eqref{eq:EL outside support U ne 0}
}%
equation holds 
\added[id=R1]{%
    with equality
}%
when $\hat \rho > 0$.
\added[id=O, comment=We have given some additional details on the deduction of (EL)]{%
    If the right-hand of \eqref{eq:EL outside support U ne 0} is positive, so is $\widehat \rho$ and equality holds. 
    When the right-hand side of \eqref{eq:EL outside support U ne 0} is non-positive, then $\widehat \rho = 0$. 
    Therefore, we have that
}%
\begin{equation*}
    \tag{EL$_3$}
    \label{eq:Euler-Lagrange}
    \hat \rho = \Bigg( (U')^{-1} \Big( h - V - W* \hat \rho \Big) \Bigg)_+, \qquad \text{for some constant } h \in \mathbb R,
\end{equation*}    
\deleted[id=O, comment=this line is redundant because we do it in the previous section]{The argument is a standard variation. We will use the optimality over a class $\hat \rho_\ee = \hat \rho + \varepsilon \varphi$ with suitably constructed  $\varphi$.}%

In particular, observe that when $U = U_m$, $V = \frac{|x|^2}{2}$ and $W = 0$ (i.e., \eqref{eq:Fokker PME}) this is the famous Barenblatt solution \eqref{eq:Barenblatt}.
More generally, the same approach holds true for the case  $W = 0$. In this case, \eqref{eq:Euler-Lagrange} is a one-parameter family $\hat \rho_h$. Furthermore, since $U'$ is non-decreasing we have that $\hat \rho_h$ is point-wise non-decreasing with $h$. Hence, in most cases $h$ can be recovered from the fixed value $\int_\Rd \widehat \rho_h$.

\begin{remark}
    For the fast diffusion case $U'(0) = \infty$. Hence, if $|h|<\infty$ and $W*\rho \in L^\infty_{loc}$ then $\widehat \rho > 0$.
\end{remark}

Notice that when $\hat \rho > 0$ then the solutions of \eqref{eq:Euler-Lagrange} are precisely stationary solutions in the sense of \eqref{eq:ADE stationary state}. When the support of $\hat \rho$ minimisers are still formally stationary, however the regularity on the free boundary (the boundary of the support) is more challenging.

\bigskip 

The cases $U = 0$ are also interesting. Take, for example \eqref{eq:Caffarelli-Vazquez} and re-scale to a Fokker-Planck problem. Then
$
    \frac{|x|^2}{2} + (-\Delta)^{-s} \widehat \rho \ge h.
$
Let $v \defeq (-\Delta)^{-s} \rho$. Hence, $(-\Delta)^s v = \hat \rho \ge 0$. 
We also get, due to \eqref{eq:Euler-Lagrange 1}, that when $\hat \rho > 0$ then equality holds in the last equation. Hence, we can write
$$
    \min\left\{ (-\Delta)^s v(x), v(x) + \frac{|x|^2}{2} - h   \right\} = 0, \qquad \text{for all } x \in \Rd.
$$
This is the well-known fractional obstacle problem. In this particular application it is discussed in \cite{CaffarelliVazquez2011b}. Regularity of solutions was proved in \cite{silvestre2007RegularityObstacleProblem}.

\bigskip 

Lastly, there is the case of the Newtonian/Riesz potential $W = - cW_k$ for $k = 2-2s \in [2-d,0)$. Then $u \defeq W * \rho = (-\Delta)^{-s} \rho$. Hence, we can re-write \eqref{eq:Euler-Lagrange} as
$$
    (-\Delta)^s u = ((U')^{-1} (u + h - V))_+
$$
The case of $U_m$ with $m > 1$, $V = 0$, and $s \in (0,1)$ was studied in \cite{chan2020UniquenessEntireGround}.
\comment[id=R1]{
    Indeed the reference \cite{carrillo+delgadino+mellet2016RegularityLocalMinimizers} was missing. 
    However, this section deals with solving the EL equations. We have cited below.
}%

\bigskip

The particular case $k \in (2-d,0)$ and $m = \frac{2d}{2d+k}$ has the correct scaling (sometimes called \emph{conformal}). The explicit self-similar solution can be found using the result in \cite{chen2006ClassificationSolutionsIntegral} and the references therein (see also \cite{bian2023AggregationdiffusionEquationEnergy}) to be
\begin{equation}
    \label{eq:steady state conformal fractional}
    \hat \rho (x) = B \left( \frac{\lambda}{\lambda^2 + |x-x_0|^2}  \right)^{\frac{d+2s}2}.
\end{equation}

\begin{remark}
    This also shows that in $\Rd$, when $U'$ is strictly increasing, $U'(0) = -\infty$, and $V,W$ is bounded, there can exist no stationary states. We compute the lower bound
    $$
        \hat \rho_h (x) \ge (U')^{-1} (  h - \| V \|_\infty - \|W \|_\infty \| \rho \|_\infty ) > 0.
    $$
    This means that, $\hat \rho_t \notin L^1(\Rd)$ is not for any $h \in \mathbb R$.
\end{remark}

\begin{remark}[Bifurcation analysis]
    Several authors have paid attention to the structure of local minimisers in terms of bifurcations. These analyses can be done for \eqref{eq:Kuramoto} in terms of Fourier coefficients. 
    In 
    \added[id=R1]{%
    \cite{carrillo2020LongTimeBehaviourPhase} 
    }%
    the authors study bifurcation of \eqref{eq:McKean-Vlasov} (with periodic boundary conditions, or equivalently in $\mathbb T^d$, with $W = \kappa \bar W$, and $V = 0$) by applying Crandall-Rabinowitz theory. They see local minimisers as solutions of the equation
    $$
        (\rho, \kappa) \mapsto \rho - \frac{e^{-\beta \kappa \bar W*\rho}  }{\int e^{-\beta \kappa (\bar W*\rho)(y)} \diff y},
    $$
    over the space of even functions. This was later extended in \cite{carrillo2021PhaseTransitionsNonlinear} to \eqref{eq:ADE} (again in the torus but with $U = U_m$ and $V = 0$) by taking the suitable extension of the energy.
\end{remark}

\begin{remark}
    When $U = 0$ the equation $V + W * \rho = h$ has been studied using different approaches, but we include these references in \Cref{sec:EL when U=0} below.
\end{remark}

	\subsection{Extension of the free energy to measure space}

Our free energy $\calF$ is defined for $\rho \in \mathcal P_2(\Rd) \cap L^1 (\Rd)$. 
This set is not dense in $\mathcal P_2(\Rd)$ with the total variation metric since for $\| \rho - \delta_0\|_{TV} = 2$.
However, it is dense in the narrow topology. We can therefore try to extend $\mathcal F$. The natural weak lower-semicontinuous extension is given by
$$
    \widetilde{\calF} (\mu) = \inf_{\rho_n \rightharpoonup \mu} \liminf_n \calF(\rho_n).
$$
Clearly, this is not so easy to compute directly.

\bigskip

Let us look first at the diffusion term in the homogeneous range $U_m$. As an extreme example, we have to make sense of $\delta_0^m$. An easy approach is to take $\rho_1 \in  \mathcal P_2(\Rd) \cap C_c^\infty (\Rd)$, define $\rho_\ee = \ee^{-d} \rho_1(\ee^{-1} x)$. Then $\rho_\ee \rightharpoonup \delta_0$. Integrating we recover
$$
    \calU_m (\rho_\ee) = \ee^{d(1-m)} \calU_m(\rho_1).
$$
Therefore, the extension has to be 
$$
    \widetilde{\calU_m} (\delta_0) = 
    \begin{dcases} 
        +\infty & \text{if } m > 1, \\
        0 & \text{if } m \in (0,1).
    \end{dcases} 
$$
The rigorous construction of $\widetilde{\cal U}$ in general setting can be found in
\cite{DemengelTemam1986}. In fact
$$
    \widetilde \calU(\mu) = 
    \begin{dcases} 
        +\infty & \text{if } 
        \added[id=R1]{%
            \lim_{s \to \infty }\tfrac{U(s)}{s} = \infty 
        }%
        \text{ and } 
        \added[id=R1]{%
            \mu_{\mathrm {sing}} 
        }%
        \not \equiv 0, \\
        \calU (\mu_{ac})  & \text{if } 
        \added[id=R1]{%
            \lim_{s \to \infty }\tfrac{U(s)}{s} =0,
        }%
    \end{dcases} 
$$
where 
\added[id=R1,comment=``and'' is the correct option. If there is no singular part{,} we do not need to extend the measure and the possible infinity comes from the Lebesgue integration]{%
$\mu = \mu_{\mathrm{ac}} + \mu_{\mathrm{sing}}$ 
}%
the absolutely continuous and singular parts of the measure. If $V, W$ are well-behaved (e.g., bounded by $C(1+|x|^2)$) then
$$
    \widetilde {\calF} (\alpha \delta_0 + \rho ) = \widetilde{\calU} (\rho) + \int V \diff \mu + \frac 1 2 \iint (W*\mu) \diff \mu. 
$$
For example, if $U$ is sublinear and $V, W$ are well-behaved
$$
\widetilde {\calF} (\alpha \delta_0 + \rho ) = \calF(\rho) + \alpha V(0) + 2 \alpha \int_\Rd W(x) \rho(x) \diff x + \alpha^2 W(0).
$$
The first variation in this setting is a little more involved, but it follows the same general scheme. 

\bigskip

However, if $W$ is singular at $0$ the landscape can be richer, for example as mentioned in \Cref{rem:scaling free energy KS} for \eqref{eq:Keller chi}, the correct behaviour extension is more difficult to construct.
In fact, we point the reader to \cite{delgadino2022UniquenessNonuniquenessSteady} where the authors construct example of $W$ with infinitely many radially decreasing steady states.
	\subsection{Existence of minimisers}
\label{sec:existence of minimisers}

\subsubsection{General comments}
For example, it is easy to show that when $m > 1$, and $V, W$ are $\lambda$-convex for $\lambda > 0$ and bounded below, then the free-energy functional admits a unique minimiser with the properties above using the direct method of calculus of variations. 
First, the functional is bounded below since 
\added[id=R1,comment=Reply to referee: in this paragraph we are discussing $m > 1$]{%
    $U_m \ge 0$ for $m > 1$.     
}%
We take a minimising sequence $\rho_n$, and we use Prokhorov's theorem to prove it has a narrow limit $\hat \mu$. Using weak lower semi-continuity we show that the limit is indeed a minimiser.

The situation in bounded domain $\Omega$ is easier, since all sequences of measures are tight (i.e., mass cannot escape to infinity) and the Lebesgue spaces are embedded in each other. We focus therefore on the more complicated case $\Rd$. 

\bigskip

To find energy minimisers in $\Rd$, one can try to construct minimising profiles by scaling.
Given a profile $\rho_1$, we can rescale it like 
$$
    \rho_\lambda (x) = \lambda^{d} \rho_1( \lambda x ).
$$ 
It might happen that \emph{full diffusion} (i.e., $\lambda \to 0$) is energy beneficial.
Indeed, we have the scaling
$$
    \calU_m (\rho_\lambda) = \begin{dcases} 
        \lambda^{d(m-1)} \calU_m (\rho_1) &\text{if } m \ne 1 \\
        \calU_m(\rho_1) + \log \lambda & \text{if } m = 1 ,
    \end{dcases} 
$$
Notice that, when $m > 1$ then $d(m-1)$ and $\calU_m(\rho_1) > 0$, whereas if $m < 1$ they are both negative. Furthermore, we observe that, when $m > 1$, then the energy is bounded below; and for $m \in (0,1]$ it is not.

\bigskip

It might also happen that \emph{full concentration} (i.e., $\lambda \to \infty$) is best for the energy, for example for $U, W = 0$ and $V = |x|^2$. Then we expect the minimiser to be a $\delta_0$.

\subsubsection{The case \texorpdfstring{$U = 0$}{U=0}}
\label{sec:EL when U=0}

There is a long literature looking for minimisers when $U = 0$. 
In this setting, some authors have studied the $\Wass_p$ local minimisers with $p \in (1,\infty]$. In \cite{balague2013DimensionalityLocalMinimizers} the authors realised that, when $V = 0$, the minimisers can be supported over sets of different dimensions, depending on the attractive-repulsive nature of $W$.
The stability of these minimisers was later studied in \cite{balague2013NonlocalInteractionsRepulsive}.
\added[id=R1]{%
    In \cite{balague2014ConfinementRepulsiveattractiveKernels} 
    the authors discuss
    the compact support of minimisers under general settings.
}%
For general conditions on existence of global minimisers see \cite{simione2015ExistenceGroundStates}.
\added[id=R1]{%
    The regularity of compactly supported $\infty$-Wasserstein when $V = 0$ was studied in \cite{carrillo+delgadino+mellet2016RegularityLocalMinimizers}.
}%
Recently, analysis of the Fourier transform has been applied to determine further structure of minimisers \cite{carrillo2022Minimizers3DAnisotropic,carrillo2022RadialSymmetryFractal}. 

\bigskip 

There is the particularly interesting case of $V = 0$ and the power-type attractive-repulsive potential
\begin{equation} 
    \label{eq:W attractive-repulsive}
    W(x) = \frac{|x|^\gamma}{\gamma} - \frac{|x|^\alpha}{\alpha}.
\end{equation} 
In this direction, there have a number of works.
The authors of \cite{carrillo2014GlobalMinimizersRepulsive} show existence of minimisers coming as a limit of empirical measures 
\added[id=R1]{%
    (recall \eqref{eq:empirical distribution})
}%
when $\gamma > \alpha$.
In dimension 1 one can construct solutions of $W*\rho = E$ by inverting Fredholm operators (see \cite{carrillo2017ExplicitEquilibriumSolutions}).
Nevertheless, the measures
$$\mu_{\mathfrak m} = \mathfrak m \delta_0 + (1-\mathfrak m) \delta_1$$
are of special relevance. 
In \cite{kang2021UniquenessCharacterizationLocal} the authors show in $d = 1$ if $\alpha \ge 2$ and $\gamma$ is large enough, then $\mu_{\frac 1 2}$ is unique $\mathfrak W_p$ minimiser for $p \in [1,\infty)$.
The $\mathfrak W_\infty$ minimisation when $d = 1$ is richer:
\begin{itemize}
    \item If $\gamma > \alpha > 2$, for every $\mathfrak m \in  (0, 1)$, $\mu_{\mathfrak m}$ is a ${\mathfrak W}_\infty$-strict local minimiser. 
    \item  If $\gamma > 3, \alpha = 2$, for every $\mathfrak m\in (\frac 1 {\gamma-1} , \frac{\gamma-2} {\gamma-1})$, $\mu_{\mathfrak m}$ is a ${\mathfrak W}_\infty$-strict local minimiser. 
    \item  If $\gamma > 3, \alpha = 2$,for every $\mathfrak m \in (0, \frac1 {\gamma-1} ] \cup [ \frac{\gamma-2}{\gamma-1} , 1)$, $\mu_{\mathfrak m}$ is a ${\mathfrak W}_\infty$-saddle point. 
    \item If $3 > \gamma > \alpha = 2$, for every $\mathfrak m\in (0, 1)$, $\mu_{\mathfrak m}$ is a ${\mathfrak W}_\infty$-saddle point. 
    \item  If $\gamma = 3, \alpha = 2$, for every $m\in (0, \frac 1 2 ) \cup (\frac 1 2 , 1)$, $\mu_{\mathfrak m}$ is a ${\mathfrak W}_\infty$-saddle point. 
    \item  If $\gamma = 3, \alpha = 2$, $\mu_{\frac 1 2}$ is a ${\mathfrak W}_\infty$-strict local minimiser.
\end{itemize}
When $d > 1$ the suitable extension of $\rho_{\mathfrak m}$ is a ``Dirac delta'' supported on the surface of a ball, see \cite{daviesClassifyingMinimumEnergyI,daviesClassifyingMinimumEnergyII}. 
The explicit shape of the $\mathfrak W_\infty$-minimiser when $\gamma \in (2,3)$ and $\alpha = 2$ is given in \cite{frank2022MinimizersOnedimensionalInteraction,frankMinimizersAggregationModel}.
\added[id=R1]{%
    There is also significant interest in the case $\gamma = 2, \alpha \in (-2,0)$, see \cite{carrillo2024GlobalMinimizersLarge} and the references therein.
}%

    \bigskip

\added[id=R1]{%
    There are several related problems with different choices of $V,W$ which have been studied over the last few years due to the modelling, we point the reader to \cite{carrillo2020EllipseLawKirchhoff,carrillo2021EquilibriumMeasureAnisotropic,mateu2023ExplicitMinimisersAnisotropic}, and the references therein.
}%

\subsubsection{The power cases \texorpdfstring{$U = U_m$}{U=Um}, \texorpdfstring{$V=0$}{V=0}, and \texorpdfstring{$W = \chi W_k$}{W=chiWk}}
\label{sec:minimisation of power cases}

The family of cases $U_m (\rho) = \frac{\rho^m}{m-1}$, $V = 0$, and
$W = \chi W_k$ has been studied in extensive detail.
We analyse the scaling of the second term of the energy to recover
$$
    \calW_k(\rho_\lambda) = 
        \begin{dcases} 
            \lambda^{-k} \calW(\rho_1)&\text{if } k \ne 0 \\
            \calW_k(\rho_1) - \log \lambda & \text{if } k = 0 ,
        \end{dcases}  
$$
The scaling of the energy is given by
$$
    \calF_{m,\chi,k}(\rho_\lambda) = \lambda^{-k} \left( \lambda^{d(m-1) + k} \calU_m (\rho_1) + \chi \calW(\rho_1) \right)
$$
The sign inside the parenthesis is crucial. As pointed out in \cite{calvezEquilibriaHomogeneousFunctionals2017,CarrilloHoffmannMaininiVolzone2018} the two terms are in balance when $d(m-1) = k$, and this leads to the critical value
\begin{equation} 
    \label{eq:ADE critical exponent}
    m_c = \frac{d-k}{d}
\end{equation}
called the \textit{fair-competition regime}. 
For $m > m_c$ we have the so-called \textit{diffusion-dominated regime} and for $m \in (0,m_c)$ the \textit{aggregation-dominated regime}.
Notice that $k > 0$ implies $m_c<1$ (fast-diffusion) whereas $k < 0$ implies $m_c>1$ (slow-diffusion).

\begin{remark}
    The case scaling argument can be repeated for $\calF = \calU_m + \chi\mathcal V_\lambda$, and we recover also the critical value $m_c = \frac{d-\lambda}{d}$. The critical value of \eqref{eq:Fokker PME} corresponds precisely to $\lambda = 2$.
\end{remark}

\begin{remark}
    \label{rem:free energy minimisation Keller PME}
    Notice that \eqref{eq:Keller PME} corresponds precisely to $\lambda = 2 - d$, so we recover $m_c = \frac{2d - 2}{d}$.
    The case $U = U_m$ with $m \ge 2$, $V = 0$ and $W$ radially decreasing was studied in \cite{bedrossian2011GlobalMinimizersFree}, where the main tool is radially decreasing rearrangement of $\rho$ (which will be presented in \Cref{sec:radial symmetry}). 
\end{remark}

\begin{remark}[Hardy-Littlewood-Sobolev] To show that the free energy functionals are bounded below, we recall the Hardy-Littlewood-Sobolev inequality $k \in (-d,0) $ and $\rho \in \mathcal P(\Rd)$
\begin{equation}
    \iint_{\Rd \times \Rd} |x-y|^\lambda \rho(x) \rho(y) \diff x \diff y \le C_{\rm{HLS}}(k,d) \int_\Rd \rho(x) ^{m_c} \diff x,
\end{equation}
where we recall the definition of $m_c$ given in \eqref{eq:ADE critical exponent}.
A proof can be found in \cite[Theorem 3.1]{calvezEquilibriaHomogeneousFunctionals2017}
The logarithmic version is that if $\log(1+|\cdot|^2) \rho \in L^1$ and $\rho \in \mathcal P(\Rd)$ then
\begin{equation}
    - \iint_{\Rd \times \Rd} \log |x-y| \rho(x) \rho(y) \diff x \diff y \le \tfrac 1 d \,\calU_1(\rho) + C_0 .
\end{equation}
\end{remark}

The fair-competition regime was discussed in \cite{calvezEquilibriaHomogeneousFunctionals2017}, the diffusion-dominated regime in \cite{CarrilloHoffmannMaininiVolzone2018}, and the aggregation-dominated regime in \cite{CarrilloDelgadinoPatacchini2019}. 

Since there is fair competition, the rescaling done for \eqref{eq:PME} suitably rescales the equation. There is, as usual, a new term $\diver(x \rho)$ coming from the time derivative. This new equation is itself the 2-Wasserstein gradient flow of the rescaled free energy
$$
    \calF_{\mathrm{resc}} \defeq \calF + \mathcal V_2,
$$
where we recall the definition \eqref{eq:free energy parameters}.
The main results are as follows
\begin{theorem}[Fair-competition case \cite{calvezEquilibriaHomogeneousFunctionals2017}]
    Let $m = m_c$ and $k \in (-d,0)$. Then
    \begin{enumerate}
       \item Stationary states satisfy $\calF(\hat \rho) = 0$.
       \item We have that 
       $$
        \calF (\rho) \ge \frac{1 - \chi C_{\rm{HLS}}(k,d)}{d(m-1)} \|\rho\|_{L^m}^m.
       $$
       Due to the previous item, we define $\chi_c = 1/C_{\rm{HLS}}(k,d)$.

       \item If $\chi = \chi_c$, 
       \added[id=O, comment=we have unified the grammar]{%
       then 
       }%
       there exists a global minimiser.
       \item If $\chi < \chi_c$, 
       \added[id=O]{%
       then 
       }%
        there exists a global minimiser for $\calF_{\mathrm{resc}}$, but not for $\calF$.
       \item If $\chi > \chi_c$, then $\mathcal F$ and $\calF_{\mathrm{resc}}$ are not bounded below.
    \end{enumerate}
\end{theorem}

\begin{theorem}[Diffusion-dominated regime \cite{CarrilloHoffmannMaininiVolzone2018}]
    Let $m > m_c, d \ge 1, \chi > 0,$ and $k \in (-d,0)$. Then there exists a global minimiser $\hat \rho \in L^1_+(\Rd) \cap L^m(\Rd)$, and it is radially symmetric and compactly supported.
\end{theorem}

The work \cite{CarrilloHittmeirVolzoneYao2019} also contains contributions towards existence of minimisers in the diffusion-\-do\-mi\-na\-ted range. However, the key contribution is the radial symmetry as described in \Cref{sec:radial symmetry}.
More or less in parallel, work was done on the aggregation-dominated regime.

\begin{theorem}[Aggregation-dominated regime \cite{CarrilloDelgadinoPatacchini2019}]
    Assume that $W \in L^1_{loc} (\Rd)$ and $W(x) = W(-x)$. 

    \begin{enumerate}
        \item (Linear diffusion) Assume, furthermore, that $W \in L^\infty(\Rd \setminus B_\delta)$. Then, for any $\ee > 0$ the energy $\ee \calU_1 + \calW$ does not admit any $\Wass_p$-local minimisers. Furthermore, if $W$ is Lipschitz continuous, then it does not admit any critical points. 
        
        \item (Fast diffusion) If $m < \frac{d}{d+k}$ then the free energy $\ee U_m + \calW_k$ is not bounded below in $\mathcal P(\Rd) \cap L^\infty (\Rd)$.
        
        \item (Slow diffusion) If $m > 1$ and $k \in ((1-m)d , 0)$ then the energy
        $$
            \cal F(\rho) = \begin{dcases}
                \calU_m(\rho) + \calW_k (\rho) & \text{if } \rho \in L^m (\Rd) \\
                +\infty & \text{if } \rho \in \mathcal P(\Rd) \setminus L^m (\Rd)
            \end{dcases}
        $$  
        has a global minimiser.
    \end{enumerate}
\end{theorem} 
The authors also provide a condition for sharp existence of global minimisers. 
In \cite{carrillo2019ReverseHardyLittlewood} the authors introduce a reversed HLS inequality to deal with the aggre\-gation-domi\-nated regime.

\begin{theorem}[Reversed HLS \cite{carrillo2019ReverseHardyLittlewood}]
    \begin{enumerate}
        \item If $k > 0$ and $m \in ( \frac{d}{d+k} ,1)$ then there exist a constant $C \in \mathbb R$ such that
        \begin{equation} 
            \label{eq:reversed HLS}
        \dfrac{ \iint_{\Rd \times \Rd} f(x) |x-y|^k f(y) \diff x \diff y}{\left( \int_\Rd f(x) \diff x \right)^{\alpha} \left( \int_\Rd f(x)^m \diff x \right)^{\frac{2-\alpha}{m}}} \ge C  , 
        \end{equation} 
        where
        \[
        \alpha = \frac{2d-m(2d+k)}{d(1-m)} .
        \]
        \item The equality is in \eqref{eq:reversed HLS} achieved if $d=1,2$ or $d = 3$ and $m \ge \min\{ \frac{d-2}{d} , \frac{2d}{2d+k}   \} $.
        \item If $m \in ( \frac{d}{d+k} ,1)$ and either $\lambda \in [2,4]$ or, $\lambda \ge 1$ and $m \ge \frac{d-1}{d}$ then the minimiser of \eqref{eq:reversed HLS} is unique (up to translation, dilation, or multiplication). 
    \end{enumerate}
\end{theorem} 

Part of the results in \cite{carrillo2019ReverseHardyLittlewood} appeared first as a preprint \cite{CarrilloDelgadino2018} written in terms of \eqref{eq:ADE} instead of functional inequalities. We present the results as stated there as a corollary of the published paper.
\begin{corollary}
    \label{cor:Carrillo Delgadino}
    Let $m \in( \frac{d}{d+k} , 1)$ and $k > 0$. Then
    \begin{enumerate}
        \item there exists a unique $\widehat \mu$ that minimises the free energy $\widetilde \calU_m (\widehat \mu) + \chi \calW (\widehat \mu)$, and they coincide with minimisers of \eqref{eq:reversed HLS}.
        
        \item It is of the form \eqref{eq:steady state with concentrated part}
        for some $\widehat \rho \in L^1 $. 

        \item For $m > \frac{d-1}{d}$ and $k \ge 1$ the functional is geodesically convex in $\Wass_2$, and hence, the minimiser is unique.
    \end{enumerate}
\end{corollary}
The authors also present the suitable Euler-Lagrange equation for $\widehat \mu$. The computation is more involved that \eqref{sec:first variation probability} that follows the same philosophy. 

\bigskip

The case $\lambda = 4$ was studied in full detail in \cite{CarrilloDelgadinoFrankLewin2020}. Here the authors characterise precisely the dimensions $d$ and index $m$ for 
\added[id=R1]{%
has a Dirac delta, i.e., it
is \eqref{eq:steady state with concentrated part} with $\|\widehat \rho\|_{L^1} < 1$.
}%

\begin{remark} 
A particular amount of work has been devoted to the case $V = 0$ whereas $V \ne 0$ raised much lower interest. The extension of the results is more or less direct, see \cite{Carrillo+GC+Vazquez2022JMPA,Carrillo+Fernandez-Jimenez+GC2023}.
\end{remark} 

\begin{remark}
    \label{rem:scaling free energy KS}
    Keller-Segel corresponds precisely to the logarithmic cases $d = 2$, $m = 1$, and $W = \frac {-1} {2 \pi} W_0 $ which yield (assuming $\| \rho \|_{L^1} = 1$)
    $$
    \calF [ \rho_\lambda] = \calF(\rho_1) + \left(2 -  \frac{ \chi}{4 \pi}\right) \log \lambda.
    $$
    This leads to the critical value of $\chi^* = 8 \pi$ or, equivalently recalling the change of variable in \Cref{sec:Keller-Segel}, critical mass $M^* = 8 \pi$. 
\end{remark}

	\subsection{Radial symmetry}
\label{sec:radial symmetry}
There are several approaches to prove that, when $V, W$ are radially symmetric then the minimisers must be radially symmetric. 
One approach is to use the associated stationary problem, and Alexandrov reflexions. 

A different approach is to prove that we can find radially symmetric minimising sequences. 
For this, we can take advantage of the rearrangement theory (see \cite{Talenti2016} and the references therein).
First, let us ``re-arrange'' a set. If $A \subset \Rd$ we define the rearrangement of $A$ as the ball of radius $0$
$$  A^\star \defeq B(0,R) \qquad \text{ such that } |A^\star| = |A|.$$
We define the \textit{radially decreasing re-arrangement} of a non-negative function $f$
as the function $f^\star$ such that
$$
    \{ x : f^\star (x) \ge t \} = \{ x : f(x) \ge t \}^\star.
$$
We can write $f$ in terms of its level sets via the \textit{layer-cake representation}
$$
    f(x) = \int_0^{f(x)} \diff t = \int_0^\infty 1_{[0,f(x)]} (t) \diff t = \int_0^\infty 1_{L_t^+(f)} (x) \diff t,
$$
where $1_A$ is the indicator function of $A$, and we have introduced the super-level set notation
$$L_t^+ (f) \defeq \{ x : f(x) \ge t \}.$$
Hence, we can simply define
$$
    f^\star(x) \defeq \int_0^\infty 1_{L_t^+(f)^\star} (x) \diff t.
$$

It is not difficult to check that 
\added[id=R1,comment=the hypothesis of {$f,g,h$} have been precised. We also give the names of the formulas]{%
if $f,g,h$ are non-negative functions vanishing at infinity we have the conservation of all $L^p$ norms, which can be more generally written as
\[
    \int_\Rd U(f^\star(x)) \diff x = \int_\Rd U(f(x)) \diff x, \qquad \text{for all }U\text{ convex}.
\]
{Furthermore, we have  Hardy-Littlewood's inequality}
\[
    \int_\Rd f^\star(x) g^\star(x) \diff x \ge \int_\Rd f(x) g(x) \diff x,
\]
{and Riesz's inequality}
\[
    \iint_{\Rd\times\Rd}  f^\star(x) g^\star(y) h^\star(x-y) \diff x \diff y \ge \iint_{\Rd \times \Rd}  f(x) g(y) h(x-y) \diff x \diff y.
\]
    By reverting these inequalities, if $V, W$ are non-negative and radially increasing we have that
    $$
        \calF[\rho^\star] \le \calF[\rho].
    $$
    Hence, the expected behaviour is that if $\hat \rho$ is a minimiser, 
    then so is its radially decreasing rearrangement. 
}%
Furthermore, any minimising sequence $\rho_k$ can be replaced by a radially-decreasing minimising sequence.
However, rigorously proving this intuition is rather technical. A fairly general theorem can be seen in \cite{CarrilloHittmeirVolzoneYao2019}, where the authors take advantage of the continuous Steiner rearrangement.

\section{Numerical methods}
	
	\subsection{Finite Volumes}

Finite volume schemes are a family of schemes for conservation equations which incorporate the idea of transport of mass.
They are easiest to introduce $d = 1$. 
Take a spatial mesh $x_i = i \Delta x$ we aim to approximate numerically the averages
$$
	{\rho_i} \approx \frac{1}{x_{i+\frac 1 2} - x_{i-\frac 1 2}} \int_{x_{i-\frac 1 2}}^{x_{i+\frac 1 2}} \rho(x) \diff x.
$$
The conservation equation \eqref{eq:conservation law} yields
\begin{equation*}
	\frac{\diff}{\diff t} \int_{x_{i-\frac 1 2}}^{x_{i+\frac 1 2}} \rho(t,x) \diff x= -\Big( F(x_{i+\frac 1 2}) - F(x_{i-\frac 1 2})\Big).
\end{equation*}
Discretising in time we arrive at the general form
\begin{equation*}
	\frac{\rho_i^{n+1} - \rho_i^n }{\Delta t} + \frac{F_{i+\frac 12}^{n+1} - F^{n+1}_{i-\frac 1 2}}{\Delta x} = 0.
\end{equation*}
The no-flux condition is approximated by $F_{-N-\frac 1 2} = F_{N + \frac 1 2} = 0$, so $i = -K, \cdots, K$.
For transport problems in which $\mathbf F = \rho \mathbf v$, 
the flux is taken up-wind to preserve positivity of $\rho_i^n$, namely
\begin{equation}
	\label{sec:numerics upwinding}
	F_{i+\frac 1 2 } = \rho_i (v_{i+\frac 1 2})_+  + \rho_{i+1} (v_{i+\frac 1 2})_-,
\end{equation}
where $u_+ = \max\{u,0\}, u_- = \min\{u,0\}$ so that $v = v_+ + v_-$.
For \eqref{eq:ADE} the velocity field $v_{i+\frac 1 2}$ is an approximation of the velocity field
\begin{equation*}
	v_{i+\frac 1 2} \approx -\nabla (U'(\rho) + V + W*\rho) (x_{i+\frac 1 2}).
\end{equation*}

In \cite{BailoCarrilloHu2020}  the authors introduce a scheme with energy decay for \eqref{eq:ADE}. When $W = 0$, if $U$ is concave then so is $E(x,\rho) = U(\rho) + V(x)\rho$ and hence we can proof a version of \eqref{eq:free energy dissipation} with an inequality
\begin{align*}
	\frac 1 {\Delta t} \Bigg( \sum_i E(x_i, \rho_i^{n+1}) &- \sum_i E(x_i, \rho_i^{n}) \Bigg) \\
	&\le \sum_i \frac{\partial E}{\partial \rho}(x_i, \rho_i^{n+1}) \frac{\rho_i^{n+1} - \rho_i^n }{\Delta t}  \\ 
	&\le - \sum_i \frac{\partial E}{\partial \rho}(x_i, \rho_i^{n+1}) \frac{F_{i+\frac 12}^{n+1} - F^{n+1}_{i-\frac 1 2}}{\Delta x} \\
	&= \sum_i \frac{ \frac{\partial E}{\partial \rho}(x_{i+1}, \rho_{i+1}^{n+1}) -  \frac{\partial E}{\partial \rho}(x_i, \rho_{i}^{n+1}) }{\Delta x}    F_{i+\frac 12}^{n+1} .
\end{align*}
Due to the convexity tricked used above, it is natural to take the implicit method
\begin{align*}
	v_{i+\frac 1 2}^{n+1} &= -\frac{ \frac{\partial E}{\partial \rho}(x_{i+1}, \rho_{i+1}^{n+1}) -  \frac{\partial E}{\partial \rho}(x_i, \rho_{i}^{n+1}) }{\Delta x}\\
	F_{i+\frac 1 2 }^{n+1} &= \rho_i^{n+1} (v_{i+\frac 1 2}^{n+1})_+  + \rho_{i+1}^{n+1} (v_{i+\frac 1 2}^{n+1})_-.
\end{align*}
Then we have the energy decay
\begin{align*}
	\frac 1 {\Delta t} \Bigg( \sum_i E(x_i, \rho_i^{n+1}) &- \sum_i E(x_i, \rho_i^{n}) \Bigg) \le - \sum_i \left(\rho_i^{n+1} ( v_{i+\frac 1 2}^{n+1})_+^2 + \rho_{i+1}^{n+1} (v_{i-\frac 1 2}^{n+1})_-^2  \right)  \le 0
\end{align*}
When $W$ is introduced similar estimates are possible, although the bilinear terms need delicate handling. In
 \cite{BailoCarrilloHu2020} the authors prove that a successful scheme is
\begin{subequations} 
	\label{eq:Bailo scheme}
\begin{align}
	\label{eq:Bailo scheme xi}
	\xi_{i}^{n+1} &= U'(\rho_i^{n+1}) + V(x_i) + \Delta x\sum_{j=1}^n W(x_i - x_j) \frac{\rho_j^{n+1} + \rho_j^n}{2} \\
	v_{i+\frac 1 2}^{n+1} &= - \frac{\xi_{i+1}^{n+1} - \xi_i^{n+1}}{\Delta x}, \\
		F_{i+\frac 1 2 }^{n+1} &= \rho_i^{n+1} (v_{i+\frac 1 2}^{n+1})_+  - \rho_{i+1}^{n+1} (v_{i+\frac 1 2}^{n+1})_-,
\end{align}
\end{subequations} 
with the corresponding boundary no-flux conditions. This scheme has energy decay for any $W$ even. 
The authors of \cite{BailoCarrilloHu2020} also prove that under certain assumptions on $W$ the decay of the energy holds for other discretisations without the midpoint $({\rho_j^{n+1} + \rho_j^n})/{2}$.
And a proof of convergence through compactness for the cases $m > 2$ was later introduced in \cite{BailoCarrilloMurakawaSchmidtchen2020} convergence.
This scheme can be adapted to \eqref{eq:ADE Omega} by 
\begin{equation*} 
	\label{eq:Bailo scheme extended}
		\tag{\ref{eq:Bailo scheme xi}'}
	\begin{aligned}
		\xi_{i}^{n+1} &= U'(\rho_i^{n+1}) + V(x_i) + \Delta x\sum_{j=1}^n K(x_i, x_j) \frac{\rho_j^{n+1} + \rho_j^n}{2}.
	\end{aligned}
\end{equation*}

A very similar scheme for \eqref{eq:AE} is studied in \cite{delarue2020ConvergenceAnalysisUpwind}, where the authors prove convergence in Wasserstein sense even for pointy potentials. This scheme is later extended for linear diffusion in \cite{lagoutiereVanishingViscosityLimit}.

\bigskip

A higher order method for the case $U, V = 0$ and $W$ Lipschitz was constructed in \cite{carrillo2021SecondorderNumericalMethod}. The authors prove convergence in 1-Wasserstein norm.

\begin{remark}
\added[id=O, comment = we have added this comment which can be useful to the reader]{%
Some finite-volume methods have been studied actually correspond to the gradient of a discrete energy in a discrete version of the Wasserstein metric, see \cite{reda201715InterpretationFinite}.
}%
\end{remark}

\begin{remark}
	This kind of method have been generalised to $\partial_t \rho =  \diver ( m(\rho) \nabla \frac{\delta F}{\delta \rho})   )$ when $m$ is concave (see \cite{BailoCarrilloHu2021}).
\end{remark}

	\subsection{Methods based on discretising the Wasserstein distance}

There are several possibilities of numerical computing the Wasserstein gradient flow. One could combine a method numerically computing the Wasserstein distance (see, e.g., \cite{peyre2020ComputationalOptimalTransport}) with a numerical optimisation in the JKO scheme.
A very interesting paper in a similar direction is
\cite{carrillo2022PrimalDualMethods}, where the authors combine of the JKO scheme with the Benamou-Brenier formula (recall \eqref{eq:Benamou-Brenier} above) to construct a discrete gradient-flow method, which they call \emph{primal-dual method}.

\bigskip 

A different approach is to use a modified energy for which the gradient flow is easier to compute. This leads to the particle/blob methods we will discuss in the next section.
	\subsection{Particle/Blob methods}

When $U = 0$
\added[id=R1]{%
	and $V, W$ are Lipschitz,
}%
\eqref{eq:ADE} admits exact solutions given by empirical distributions
$$
	\mu_t^{(K)} = \sum_{i=1}^K w_i \delta_{X_t^{(i)}}.
$$
Hence, the PDE can be reduced to a finite set of ODEs, that can be solved by any method for ODEs.
If we are given $\rho_0 \in L^1$, then for any $K$ we approximate it 
\added[id=R1]{%
	by an empirical measure $\mu_0^{(K)}$.	
}%

\bigskip 

When $U \ne 0$ this is no longer the case, because the PDE is ``regularising''. Furthermore, the free of empirical measures is not defined for $U = U_m, m \ge 1$.
However, there are two approaches to approximate the free energy so that the $\Wass_2$-gradient flow admits again these kinds of solutions. 

\subsubsection{The particle method}

A first approach was introduced in \cite{CarrilloPatacchiniSternbergWolansky2016} and perfected in \cite{carrillo2017NumericalStudyParticle}. 
Fix an initial mass
$$
	\mu_0^{(K)} = \sum_{i=1}^K w_i \delta_{X_0^{(i)}}.
$$
Fix these weights of the particles and $K$ the number of particles. Take
\added[id=R1]{%
	the set of empirical distributions of $K$ elements and fixed weights
}%
$$
	\mathcal A_{K,w} (\Rd) = \left\{ \mu = \sum_{i=1}^K w_i \delta_{x_i} : \text{ for some } x_i \in \Rd \right\}.
$$
The aim is to approximate the free energy $\cal F$ so that
\begin{equation} 
	\begin{aligned} 
	\calF_K&\left[\sum_{i=1}^K w_i \delta_{x_i}\right] \defeq \calF\left[ \sum_{i=1}^K w_i \frac{\indicator{B(x_i,r_i)}}{|B(x_i,r_i)|}   \right]  \\
	&= \sum_i |B(x_i,r_i)|^d U\left( \frac{w_i}{|B(x_i,r_i)|} \right) + \sum_{i=1}^K w_i V(x_i) + \frac 1 2 \sum_{i,j=1}^K w_i w_j W(x_i - x_j),
	\end{aligned}
\end{equation} 
for suitably selected radii. Then the scheme is simply to apply a suitable adaptation of the JKO scheme \eqref{eq:JKO} with time-step $\tau$, that is
$$
	\mu_{n+1} \defeq \argmin_{\mu \in \mathcal A_{K,w}} \left( \frac{\Wass_2(\mu, \mu_n)^2}{2} + \tau \calF_K [\mu]  \right) .
$$

\subsubsection{The blob method}

The blob approach aims to modify the energy so that solutions by particles are admissible. To preserve the gradient-flow structure
we modify the free energy by changing the term $\int U( \rho )$.  A popular idea introduced by \cite{CarrilloCraigPatacchini2019} is
\begin{equation*}
	\mathcal U_\eta [\rho]  = \int_\Rd  F \circ  (\eta * \rho)  \diff \rho, \qquad \text{where } F(s) = \frac{U( s  )}{s}
\end{equation*}
where $\eta$ is a smooth probably distribution typically close to a Dirac delta. Then, the equation becomes
\begin{equation*}
	\partial_t \rho = \diver \left(\rho \nabla \frac{\delta U_\eta}{\delta \rho}  \right)
	\qquad 
	\text{where } 
	\frac{\delta U_\eta}{\delta \rho} =  \eta * \Big( \big( F' \circ (\eta * \rho) \big ) \rho\Big ) + F \circ (\eta *  \rho).
\end{equation*}
This problem admits particle-type solutions.

\bigskip

Alternative choices of approximation of the free energy are available. See, e.g., \cite{burger2023PorousMediumEquation,carrillo2023NonlocalApproximationNonlinear}.	
	\subsection{Finite Elements}

Finite element schemes rely on the weak formulation and look for solution in Sobolev spaces. They are extremely useful and accurate for problems which have an $L^2$-gradient flow structure (e.g., the $p$-Laplacian problem instead of \eqref{eq:PME}). 
\added[id=R1, comment=added some references]{%
Nevertheless, authors have found success in this approach, e.g., for \eqref{eq:Keller-Segel un-normalised}-type problems \cite{Epshteyn2009DiscontinuousGalerkinMethods,EpshteynIzmirlioglu2009FullyDiscreteAnalysis,EpshteynKurganov2009NewInteriorPenalty} and later for more general \eqref{eq:ADE} in \cite{BurgerCarrilloWolfram2010MixedFiniteElement}. For more details on finite elements for parabolic problems see \cite{thomee2006GalerkinFiniteElement}.
}%

\section{Other related problems}

\subsection{Second-order problems}

Our family of problems is of first order, since in the particle description we are specifying the velocities of the particles. 
There is a family of models which introduce the type of fields but work in acceleration formulation
\[
	\begin{dcases}
		\frac{\diff x_t^{(i)}}{\diff t} = v_t^{(i)} \\
		\frac{\diff v_t^{(i)}}{\diff t} = -\nabla V(x_t^{(i)}) - \frac 1 N\sum_j \nabla W(x_t^{(i)}-x_t^{(j)}) + \sqrt{2\sigma} \diff \mathcal B_t
		^{
		\added[id=R1]{%
			(i)
		}%
		}
	\end{dcases}
\]
\added[id=R1]{%
	with independent Brownian motions $\mathcal B_t^{(i)}$.
	When $\sigma \ne 0$ we recover a diffusion term. 
}%
This leads to the Vlasov-Poisson-Fokker-Planck equations. We point the reader to \cite{carrillo2019PropagationChaosVlasov} and the references therein.

\subsection{Cahn-Hilliard problem}

Instead of a known potential, some authors have considered the Cahn-Hilliard type problem
$$
\frac{\partial \rho}{\partial t} = \diver(\rho \nabla (-\Delta \rho) ) - \chi \Delta \rho^m,
$$
where the diffusion is of fourth order, and the aggregation is modelled by a second order operator. This problem also has a Wasserstein gradient-flow structure. We point the reader, e.g., to 
\cite{carrillo+esposito+falco+fdezjimCompetingEffectsFourthorder} and the references therein.

\subsection{The (fractional) thin-film problem}

A family of related problem is the so-called thin-film problem and its fractional variant
\begin{equation*}
	\frac{\partial \rho}{\partial t} = \diver\left(  u \nabla ( (-\Delta)^s u + V(x)   )  \right).
\end{equation*}
These are 2-Wasserstein gradient flows, where we replace $U(\rho)$ by $[u]_{H^s}^2$ and $V$ is usually $0$ or $|x|^2/2$. Usually $s \in [0,1]$ where $s = 0$ corresponds in fact to $U_2$. We point the reader to \cite{segatti2020FractionalThinFilm} and the references therein.

\subsection{Cross-diffusion systems}

In our model, all particles are of the same ``type''. Hence, there is only need for one density. If one considers models several ``types'' of particles, we will arrive at a system of PDEs. 
A similar approach to ours can be done in the so-called \textit{cross-diffusion systems}. They correspond to the suitable Wasserstein flow of several-variable energies like
\begin{equation*}
	\calF[\rho_1, \rho_2] = \int_\Rd U(\rho_1, \rho_2) + \frac 1 2 \int_\Rd \rho_1 W_1 * \rho_1 + \frac 12 \int_\Rd \rho_2 W_2 \rho_2 + \int_\Rd \rho_1 K*\rho_2 .
\end{equation*}
See, e.g., 
\added[id=R1]{%
	\cite{carrillo2018ZoologyNonlocalCrossDiffusion,DiFrancescoEspositoFagioli2018,bur2020SegregationEffectsGap,chendausjungel2019RigorousMeanfieldLimit}	
}%
and the references therein.

\subsection{Non-linear mobility}

Several authors have studied the family of problems with so-called \textit{non-linear mobility} $\mathrm{m}$ given by
\begin{equation}
	\frac{\partial \rho}{\partial t} = \diver\left( \mathrm{m} (\rho) \nabla \frac{\delta \calF}{\delta \rho} \right).
\end{equation}
This family of problems correspond to a Wasserstein-type gradient flow, with a modified distance. The correct way to modify the Wasserstein distance is through an adapted Benamou-Brenier formula (recall \eqref{eq:Benamou-Brenier}). The details can be found in 
\added[id=O,comment=some previous references have been added]{%
	\cite{dolbeault2009NewClassTransporta,lisini2010ClassModifiedWasserstein,CarrilloLisiniSavareSlepcev2010}.
}%
It turns out this is only a valid distance if $\mathrm{m}$ is concave.

\subsubsection{The Caffarelli-Vázquez problem with non-linear mobility}

Some authors developed a theory for the \eqref{eq:Caffarelli-Vazquez} with non-linear mobility of the form $\mathrm{m} (\rho) = \rho^{m-1}$ with $m > 0$.
The motivation for came as a fractional generalisation of \eqref{eq:PME} with non-local pressure. A theory of existence and properties of weak solutions was developed in \cite{stan2016FiniteInfiniteSpeed,stan2019ExistenceWeakSolutions}. 
Recently, it was shown that in $d=1$ one can analyse a suitable equation for the mass, for which there is a monotone and convergent scheme
\cite{delteso2023ConvergentFiniteDifferencequadrature}.
In fact, this numerical method for the mass uses an up-winding given for $\rho$ as \eqref{sec:numerics upwinding}.
This group in Trondheim also prove convergence in $\rho$ variable in the so-called Bounded Lipschitz distance (also called Fortet–Mourier distance), where we point the reader to \cite[page 97]{Villani2009}. This distance is defined as an extension of the duality formula \eqref{eq:Kantorovich-Rubinstein} as
\begin{equation*}
	d_{bL} (\mu, \nu) \defeq   \sup \left\{  \int \psi \diff \mu - \int \psi \diff \nu : \text{for } \psi \text{ such that } \|\psi\|_{L^\infty} + \| \nabla \psi \|_{L^\infty} \le 1    \right\},
\end{equation*}
where we point out the addition of the bound of $\psi$ in $L^\infty$.

\subsubsection{Newtonian vortex with non-linear mobility}
Some authors became interested in the case of non-linear mobility version of \eqref{eq:Newtonian vortex} given by
\begin{equation}
	\begin{dcases}
		\frac{\partial \rho}{\partial t} = \diver (\rho^\alpha \nabla v ) \\ 
		-\Delta v = \rho ,
	\end{dcases}
\end{equation}
The cases $\alpha \in (0,1)$ (where $\mathrm m$ is concave) are studied in \cite{CarrilloGomezCastroVazquez2022AIHP}, where the authors show that for $\alpha < 1$ there exist nice self-similar solutions. 
The cases $\alpha > 1$ \cite{CarrilloGomezCastroVazquez2022ANONA}, where the authors find self-similar solutions do not have finite mass, and in fact the attracting solution is a family of characteristic functions. 
The technique in \cite{CarrilloGomezCastroVazquez2022AIHP,CarrilloGomezCastroVazquez2022ANONA} is to find solutions by generalised characteristics of the mass equation.

\subsubsection{Saturation}
The problem \eqref{eq:ADE} with $\mathrm{m}(\rho) = \rho \psi(\rho)$ with $\psi(1) = 0$ has recently gained the attention of some experts. This addition ensures the that we can work with solutions such that $0 \le \rho \le 1$ (in particular no formation of $\delta$), and hence some authors use the term \textit{saturation} to describe this choice. 

\bigskip

The papers \cite{DiFrancescoFagioliRadici2019,Fagioli_2022} explore the case $U = 0$.
The case $U \ne 0$ is, so far, only understood numerically.
A finite-volume numerical scheme was constructed for this case \cite{BailoCarrilloHu2021}. Their numerical results suggest that steady-states will be of the form
$$
	\hat \rho = \min\{ 1, (U')^{-1}( h-V-W*\hat \rho) \}_+
$$
for suitable $h$ such that the total mass is preserved. However, there is currently no rigorous proof that this is the correct choice.

\section*{Acknowledgments}

The author would like to express his gratitude to J. A. Carrillo for providing the opportunity of a two-year post-doc at University of Oxford, where this work was started, and for introducing him to a range of aggregation-diffusion equations.
The author is grateful to J.L. Vázquez for his mentorship and many fruitful conversations.
The author is thankful to A. Fernández-Jiménez for his comments on a preliminary draft of the manuscript.

\bigskip

The author 
\added[id=O]{%
    was partially supported by RYC2022-037317-I and     
}%
PID2021-127105NB-I00 from the Spanish Government, and the Advanced Grant Nonlocal-CPD (Nonlocal PDEs for Complex Particle Dynamics: Phase Transitions, Patterns and Synchronization) of the European Research Council Executive Agency (ERC) under the European Union's Horizon 2020 research and innovation programme (grant agreement No. 883363).%
\added[id=R1,comment={The relevant references have been properly ordered}]{}%

\setlength{\emergencystretch}{1.5em}
\printbibliography

\end{document}